\documentclass[12pt,twoside]{amsart}
\usepackage[margin=1in]{geometry}

 \usepackage[colorlinks=true,citecolor=black,linkcolor=black,urlcolor=blue]{hyperref}

\newcommand{\refdoi}[1]{\href{http://dx.doi.org/#1}{\texttt{doi:#1}}}
\newcommand{\refarxiv}[1]{\href{http://arxiv.org/abs/#1}{\texttt{arXiv:#1}}}

\usepackage{amsfonts,amsmath,amssymb,amsthm,enumerate,stmaryrd,colonequals,mathrsfs,versions, youngtab} 
\usepackage[pdftex]{graphicx}
\usepackage{asymptote}
\usepackage{adjustbox} 

\usepackage[colorinlistoftodos]{todonotes} 

\theoremstyle{plain}

\newtheorem{sidethm}{Side Theorem}[section]

\newtheorem{sidecor}[sidethm]{Side Corollary}

\newtheorem{theorem}{Theorem}[section] 

\newtheorem{lemma}[theorem]{Lemma} 

\newtheorem{proposition}[theorem]{Proposition} 

\newtheorem{corollary}[theorem]{Corollary} 

\newtheorem{observation}[theorem]{Observation} 

\newtheorem{exercise}[theorem]{Exercise} 

\newtheorem{problem}[theorem]{Problem} 

\newtheorem{question}[theorem]{Question} 

\newtheorem{conjecture}[theorem]{Conjecture} 

\theoremstyle{definition}

\newtheorem{sidedef}[sidethm]{Side Definition}

\newtheorem{definition}[theorem]{Definition} 

\newtheorem{example}[theorem]{Example} 

\theoremstyle{remark}

\newtheorem{sidermk}[sidethm]{Side Remark}

\newtheorem{remark}[theorem]{Remark} 

\newcommand{\bd}{\mathbf}

\newcommand{\ZZ}{\mathbb{Z}}

\newcommand{\stab}{\operatorname{Stab}}

\newcommand{\Wlog}{without loss of generality}

\newcommand{\conj}{\overline}

\newcommand{\abs}[1]{\lvert{#1}\rvert}
\newcommand{\card}[1]{\lvert{#1}\rvert}

\newcommand{\floor}[1]{\lfloor #1 \rfloor}

\DeclareFontFamily{U}{wncy}{}
\DeclareFontShape{U}{wncy}{m}{n}{<->wncyr10}{}
\DeclareSymbolFont{mcy}{U}{wncy}{m}{n}
\DeclareMathSymbol{\Sha}{\mathord}{mcy}{"58}

\newcommand{\map}{\operatorname}
\newcommand{\mcal}{\mathcal}

\newcommand{\mathd}[1]{\[#1 \]} 

\usepackage{xspace}

\makeatletter
\newcommand{\etale}{\'etal\@ifstar{\'e}{e\xspace}}
\makeatother

\begin{document}


\includeversion{arxiv}
\excludeversion{submit}
\includeversion{disclaimer}
\excludeversion{markednewpar}
\excludeversion{ejc}


\begin{ejc}

\title{\bf Simultaneous Core Partitions: \\  Parameterizations and Sums}


\author{
Victor Y. Wang\thanks{Supported by NSF grant 1358695 and NSA grant H98230-13-1-0273.}\\
\small Department of Mathematics\\[-0.8ex]
\small Massachusetts Institute of Technology\\[-0.8ex] 
\small Massachusetts, U.S.A.\\
\small\tt vywang@mit.edu\\
}


\date{\dateline{Aug 7, 2015}{Dec 28, 2015 (provisionally)}}
\end{ejc}

\begin{arxiv}
\title[Simultaneous core partitions: parameterizations and sums]{Simultaneous core partitions: parameterizations and sums}

\date{\today}

\author[Victor Y. Wang]{Victor Y. Wang}
\address{Department of Mathematics, Massachusetts Institute of Technology, \mbox{Cambridge, MA 02139}}
\email{vywang@mit.edu}


\begin{disclaimer}\thanks{v4: added reference \cite{Zeilberger} and updated \cite{CHW}; shorter version available at Electronic J. Combin. 23(1) (2016), \#P1.4. v3: minor improvements and clarifications throughout, including implicit variant details; added some references and remarks on finite beta-sets, maximal cores, and poset method. v2: extended $z$-coordinates of Johnson to parameterize all $t$-cores, not just $(s,t)$-cores; applied this to a conjecture of Amdeberhan--Leven (v3 mentions later independent proof by Johnson); added some references. Version for the arXiv, with 34 pages, 3 figures. Shorter version with 28 pages (currently available at \url{https://www.overleaf.com/read/xmzxdgbdcpnq}), with fewer details, to be submitted for publication. Comments welcome on either.}\end{disclaimer}
\begin{submit}\thanks{Shorter 
version (currently available at \url{https://www.overleaf.com/read/xmzxdgbdcpnq}). Longer version available at \url{http://arxiv.org/abs/1507.04290}.}\end{submit}


\begin{abstract}
Fix coprime $s,t\ge1$. We re-prove, without Ehrhart reciprocity, a conjecture of Armstrong (recently verified by Johnson) that the finitely many simultaneous $(s,t)$-cores have average size $\frac{1}{24}(s-1)(t-1)(s+t+1)$, and that the subset of self-conjugate cores has the same average (first shown by Chen--Huang--Wang). We similarly prove a recent conjecture of Fayers that the average weighted by an inverse stabilizer---giving the ``expected size of the $t$-core of a random $s$-core''---is $\frac{1}{24}(s-1)(t^2-1)$. We also prove Fayers' conjecture that the analogous self-conjugate average is the same if $t$ is odd, but instead $\frac{1}{24}(s-1)(t^2+2)$ if $t$ is even. In principle, our explicit methods---or implicit variants thereof---extend to averages of arbitrary powers.

The main new observation is that the stabilizers appearing in Fayers' conjectures have simple formulas in Johnson's $z$-coordinates parameterization of $(s,t)$-cores.


We also observe that the $z$-coordinates extend to parameterize general $t$-cores. As an example application with $t \colonequals s+d$, we count the number of $(s,s+d,s+2d)$-cores for coprime $s,d\ge1$, verifying a recent conjecture of Amdeberhan and Leven.

\end{abstract}

\end{arxiv}

\maketitle

\begin{ejc}

\begin{abstract}
Fix coprime $s,t\ge1$. We re-prove, without Ehrhart reciprocity, a conjecture of Armstrong (recently verified by Johnson) that the finitely many simultaneous $(s,t)$-cores have average size $\frac{1}{24}(s-1)(t-1)(s+t+1)$, and that the subset of self-conjugate cores has the same average (first shown by Chen--Huang--Wang). We similarly prove a recent conjecture of Fayers that the average weighted by an inverse stabilizer---giving the ``expected size of the $t$-core of a random $s$-core''---is $\frac{1}{24}(s-1)(t^2-1)$. We also prove Fayers' conjecture that the analogous self-conjugate average is the same if $t$ is odd, but instead $\frac{1}{24}(s-1)(t^2+2)$ if $t$ is even. In principle, our explicit methods---or implicit variants thereof---extend to averages of arbitrary powers.

The main new observation is that the stabilizers appearing in Fayers' conjectures have simple formulas in Johnson's $z$-coordinates parameterization of $(s,t)$-cores.


We also observe that the $z$-coordinates extend to parameterize general $t$-cores. As an example application with $t \colonequals s+d$, we count the number of $(s,s+d,s+2d)$-cores for coprime $s,d\ge1$, verifying a recent conjecture of Amdeberhan and Leven.


\end{abstract}

\end{ejc}

\section{Introduction}
\label{SEC:intro}

\subsection{History and motivation}

A \emph{partition} is an infinite weakly decreasing sequence $\lambda = (\lambda_1,\lambda_2,\ldots)$ of nonnegative integers with finite \emph{size} $\card{\lambda} \colonequals \lambda_1+\lambda_2+\cdots$. \begin{arxiv}(One can also think of partitions as \emph{finite} weakly decreasing sequences of \emph{positive} integers.)\end{arxiv} The \emph{Young diagram of $\lambda$} is the set $[\lambda] = \{(r,c)\in\ZZ_{>0}^2 : c \le \lambda_r\}$, often visualized as a set of $\#[\lambda] = \card{\lambda} $ boxes in some orientation. Reflecting $[\lambda]$ about the diagonal $r=c$ gives the diagram of the \emph{conjugate} partition $\conj{\lambda} = (\conj{\lambda}_1,\conj{\lambda}_2,\ldots)$, formally defined by $\conj{\lambda}_r \colonequals \#\{c\in\ZZ_{>0}: r \le \lambda_c\}$ for $r\ge1$; we say $\lambda$ is \emph{self-conjugate} if $\lambda = \conj{\lambda}$.

A partition $\lambda$ has, associated to each square $(r,c)\in [\lambda]$, a \emph{rim hook} $\{(i,j)\in[\lambda]: i \ge r,\; j\ge c,\text{ and }(i+1,j+1)\notin[\lambda]\}$ of (positive) \emph{rim hook length} $1+(\lambda_r - r) + (\conj{\lambda}_c - c)$---the same as the \emph{hook length} of the usual \emph{hook} $\{(i,c)\in[\lambda]: i \ge r\}\cup \{(r,j)\in[\lambda]: j\ge c\}$. \begin{arxiv}(For $(r,c)\notin [\lambda]$, one could extend the notion of hooks to get \emph{negative hook length}.)\end{arxiv} Importantly, \emph{removing} a rim hook of $\lambda$ leaves the Young diagram of a smaller partition. (See Figure \ref{FIG:rim-hook}.) Note that the (finite) set of hook lengths is invariant under conjugation.

When $s$ is a positive integer, we say a partition is an \emph{$s$-core} if it has no hooks of length $s$, or equivalently no \emph{rim $s$-hooks} (\emph{rim} hooks of length $s$); following Fayers \cite{Fayers}, we denote by $\mcal{C}_s$ the set of $s$-cores, and by $\mcal{D}_s\subseteq \mcal{C}_s$ the set of self-conjugate $s$-cores. More generally, any partition $\lambda$ has a unique \emph{$s$-core} $\lambda^s\in \mcal{C}_s$, given by repeatedly removing rim $s$-hooks. To prove that this \emph{$s$-core operation} $\lambda\mapsto \lambda^s$ is well-defined, one can use the \emph{beta-sets} reviewed in Section \ref{SEC:beta-sets}, which also show that $\lambda$ is an $s$-core if and only if it has no hook lengths \emph{divisible} by $s$, unifying two common definitions of $\mcal{C}_s$. These notions are connected to representation theory, symmetric function theory, and number theory (see e.g. \cite{F1,Fayers,Cranks-t-cores-expos}).

\begin{figure}
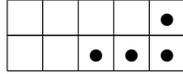

\[
\young(\ \ \ \ \bullet,\ \ \bullet\bullet\bullet)
\]
\caption{\emph{English notation} for $\lambda = (5,5,0,\ldots)$. Removing the displayed rim $4$-hook (of $(1,3)\in[\lambda]$) leaves the partition $(4,2,0,\ldots)$; removing the remaining rim $4$-hook leaves the $4$-core $\lambda^4 = (1,1,0,\ldots)$.}
\label{FIG:rim-hook}
\end{figure}

Going further, many authors (see e.g. \cite{Aggarwal,AAZ,Amol,AL,Anderson,Armstrong,AKS,CHW,F1,F2,Fayers,FV,FMS,Johnson,Olsson,OS,SZ,Vandehey,YZZ}) have recently considered the interaction of $s$-cores and $t$-cores (both the partitions and operations), for two positive integers $s,t$. For example, Anderson \cite{Anderson} showed that for coprime $s,t\ge1$, the set $\mcal{C}_s\cap\mcal{C}_t$ of \emph{(simultaneous) $(s,t)$-cores} has size equal to the number of \emph{$(s,t)$-Dyck paths}, which Bizley \cite{Bizley} had earlier enumerated---via `cyclic shifts'---as the `rational Catalan number' $\frac{1}{s+t}\binom{s+t}{t}$.\footnote{In fact, $(s,t)$-cores biject to subsets of $\ZZ_{\ge0}$ that contain $0$ and are closed under addition by $s,t$ (e.g. via \emph{beta-sets}, following negation and suitable translation), which are counted in \cite{ISL}. Through this bijection, \emph{numerical semigroups} containing $s,t$ inject into $\mcal{C}_s\cap\mcal{C}_t$.} Ford, Mai, and Sze \cite{FMS} later showed that for coprime $s,t\ge1$, the set $\mcal{D}_s\cap\mcal{D}_t$ of self-conjugate $(s,t)$-cores has size equal to the lattice path count $\binom{\floor{s/2}+\floor{t/2}}{\floor{t/2}}$. In a different direction, Olsson \cite{Olsson} showed that the $t$-core of an $s$-core is an $s$-core, hence a simultaneous $(s,t)$-core (as it is a $t$-core by definition).

In this paper, we mainly focus on related conjectures of Armstrong from \cite{Armstrong}, and Fayers from \cite{Fayers}, on certain weighted average sizes of $(s,t)$-cores when $s,t$ are coprime. Chen, Huang, and Wang \cite{CHW} established Armstrong's self-conjugate conjecture (Theorem \ref{THM:self-conjugate-Armstrong} below) using the Ford--Mai--Sze bijection \cite{FMS}. Using a poset formulation of Anderson's bijection \cite{Anderson}, Stanley and Zanello \cite{SZ} recursively established Armstrong's general conjecture (Theorem \ref{THM:general-Armstrong} below) for the `Catalan case $t=s+1$'; Aggarwal \cite{Aggarwal} generalized their method to the case $t\equiv 1\pmod{s}$. However, it is unclear whether a similar `Catalan-like' recursive structure holds for other choices of $s,t$. Recently, by different means described below, Johnson \cite{Johnson} fully proved Theorem \ref{THM:general-Armstrong}, and re-proved Theorem \ref{THM:self-conjugate-Armstrong}.

\begin{theorem}[{Armstrong \cite[Conjecture 2.6]{Armstrong}; Johnson \cite[Corollary 3.8]{Johnson}}]
\label{THM:general-Armstrong}

Fix coprime $s,t\ge1$. Then
\[ \frac{\sum_{\lambda\in \mcal{C}_s\cap \mcal{C}_t} \card{\lambda}}{\sum_{\lambda\in \mcal{C}_s\cap \mcal{C}_t} 1} = \frac{1}{24}(s-1)(t-1)(s+t+1), \]
where the sums run over \emph{all} $(s,t)$-core partitions.
\end{theorem}

\begin{theorem}[{Armstrong \cite[Conjecture 2.6]{Armstrong}; Chen--Huang--Wang \cite[proof in Section 2]{CHW}; Johnson \cite[Proposition 3.11]{Johnson}}]
\label{THM:self-conjugate-Armstrong}

Fix coprime $s,t\ge1$. Then
\[ \frac{\sum_{\lambda\in \mcal{D}_s\cap \mcal{D}_t} \card{\lambda}}{\sum_{\lambda\in \mcal{D}_s\cap \mcal{D}_t} 1} = \frac{1}{24}(s-1)(t-1)(s+t+1), \]
where the sums run over all \emph{self-conjugate} $(s,t)$-core partitions.
\end{theorem}

Developing Olsson's \cite{Olsson} and his own \cite{F1} ideas, Fayers soon after conjectured weighted analogs (Theorems \ref{THM:general-Fayers} and \ref{THM:self-conjugate-Fayers} below) of Armstrong's conjectures---in some sense giving the ``expected size of the $t$-core of a random $s$-core'' \cite{Fayers}. In this paper, we carry over explicit versions of Johnson's methods to establish both of Fayers' conjectures, despite the absence of an obvious `exponential' analog of Ehrhart reciprocity. We also briefly explain, in Remark \ref{RMK:conceptual-proof}, how one could give more implicit or ``conceptual'' proofs if necessary.

\begin{theorem}[{Fayers \cite[Conjecture 3.1]{Fayers}}]
\label{THM:general-Fayers}

Fix coprime $s,t\ge1$. Then
\[
\frac{
\sum_{\lambda\in \mcal{C}_s\cap \mcal{C}_t}
\card{\stab_{G_{s,t}}(\lambda)}^{-1}
\cdot \card{\lambda}
}{
\sum_{\lambda\in \mcal{C}_s\cap \mcal{C}_t}
\card{\stab_{G_{s,t}}(\lambda)}^{-1}
\cdot 1
}
=
\frac{1}{24}(s-1)(t^2-1), \]
where the sums run over \emph{all} $(s,t)$-core partitions, and the stabilizers are defined in terms of Fayers' `level $t$' group action on $\mcal{C}_s$ \cite{F1,Fayers} reviewed in Definition \ref{DEF:level-t-group-action}.
\end{theorem}

\begin{theorem}[{Fayers \cite[Conjecture 4.5]{Fayers}}]
\label{THM:self-conjugate-Fayers}

Fix coprime $s,t\ge1$. Then
\[
\frac{
\sum_{\lambda\in \mcal{D}_s\cap \mcal{D}_t}
\card{\stab_{H_{s,t}}(\lambda)}^{-1}
\cdot \card{\lambda}
}{
\sum_{\lambda\in \mcal{D}_s\cap \mcal{D}_t}
\card{\stab_{H_{s,t}}(\lambda)}^{-1}
\cdot 1
}
=
\begin{cases}
\frac{1}{24}(s-1)(t^2-1) &\mbox{if } t \equiv 1 \pmod{2} \\ 
\frac{1}{24}(s-1)(t^2+2) & \mbox{if } t \equiv 0 \pmod{2}
\end{cases}, \]
where the sums run over all \emph{self-conjugate} $(s,t)$-core partitions, and the stabilizers are defined in terms of Fayers' `level $t$' group action on $\mcal{D}_s$ \cite{Fayers} reviewed in Definition \ref{DEF:level-t-group-action-self-conjugate}.
\end{theorem}

\begin{remark}
As Fayers notes in \cite{Fayers}, the orbits of $G_{s,t}$ and $H_{s,t}$ are infinite, so one weights by the inverses of the finite stabilizers instead. However, by considering finite quotients acting on certain finite subsets of $\mcal{C}_s$ and $\mcal{D}_s$, he also gives the weighted averages finite probabilistic interpretations that agree exactly, not just asymptotically, with the original averages.
\end{remark}


Johnson's \emph{$z$-coordinates} parameterization of $(s,t)$-cores, and our modest extension to general $t$-cores (see Proposition \ref{PROP:change-to-z-coordinates} for general cores and in Proposition \ref{PROP:change-to-z-coordinates-self-conjugate} for the self-conjugate specialization)---which depends on a choice of $s\ge1$ coprime to $t$---plays a key role in our paper, which rests upon his cyclic shifts argument for general cores. \begin{arxiv}The key argument (reviewed in Proposition \ref{PROP:cyclic-shift}) works because the size function for $t$-cores is \emph{cyclic} (invariant under rotation of coordinates), and the $z$-coordinates are (not invariant, but still) well-behaved under rotation.\end{arxiv} However, whereas Johnson finishes by `weighted Ehrhart reciprocity' (see \cite{Ehrhart} for an introduction to `un-weighted' Ehrhart theory), we will stick to flexible direct tools---and implicit variants thereof---which can in principle evaluate sums of arbitrary powers of the partition sizes (see Remark \ref{RMK:higher-power-methods}).

The main new observation is that the stabilizer sizes appearing in Fayers' conjectures (Theorems \ref{THM:general-Fayers} and \ref{THM:self-conjugate-Fayers}) have \emph{symmetric} (or almost symmetric) formulas in the $z$-coordinates, which simply re-index the restricted counts $\card{\mcal{S}_s(\lambda)\cap(j+t\ZZ)}$ of elements in Fayers' $s$-sets $\mcal{S}_s(\lambda)$ (from \cite{F1,Fayers}). In fact, we first prove (in Propositions \ref{PROP:s-set-t-set-interaction} and \ref{PROP:faithful-invariant}) that the sets $\mcal{S}_s(\lambda)\cap(j+t\ZZ)$ underlie the tools allowing us, in Fayers' words \cite{F1}, to ``[compare] the $t$-cores of different $s$-cores'' in the first place---thus illustrating the significance of $z$-coordinates.

As an application of the cyclic shifts in the (extended) $z$-coordinates, we also parameterize and then enumerate the simultaneous $(m,m+d,m+2d)$-cores for coprime $m,d\ge1$, verifying the following recent conjecture of Amdeberhan and Leven \cite{AL} in two different ways.

\begin{theorem}[Amdeberhan--Leven, {\cite[Conjecture 3.1]{AL}}]
\label{THM:Amdeberhan--Leven--conjecture}

Fix coprime $m,d\ge1$. Then
\[
\sum_{\lambda\in \mcal{C}_m\cap\mcal{C}_{m+d}\cap\mcal{C}_{m+2d}} 1
=\frac1{m+d}\sum_{i=0}^{\floor{m/2}} \binom{m+d}{i,i+d,m-2i}.
\]
(\cite{AL} also gives an equivalent expression using $\frac1{m+d}\binom{m+d}{i,i+d,m-2i} = \binom{m+d-1}{2i+d-1} \binom{2i+d}{i}\frac1{2i+d}$.)
\end{theorem}

\begin{remark}
Following the recursive method of Stanley and Zanello \cite{SZ}, Amdeberhan and Leven \cite{AL} (and independently, Yang, Zhong, and Zhou \cite{YZZ}) proved the $d=1$ case via \emph{Motzkin number} interpretations. Both papers also proved other facts about $(m,m+1,m+2)$-cores.
\end{remark}

\begin{remark}
A few days after the arXiv postings of v2 of the present paper and v1 of \cite{AAZ}, Paul Johnson informed us (via private correspondence) that he had independently found our \emph{asymmetric proof} (in Section \ref{SEC:Amdeberhan--Leven--conjecture}) of Theorem \ref{THM:Amdeberhan--Leven--conjecture}. In fact, he proved the slightly stronger result that the term $\frac1{m+d}\binom{m+d}{i,i+d,m-2i}$ counts the number of $(m,m+d,m+2d)$-cores with exactly $i$ hooks of length $d$; see the end of Remark \ref{RMK:key-quantities-s-set-t-set-interaction} for a brief explanation.
\end{remark}

\subsection{Outline of paper}

In Section \ref{SEC:s-core-exposition}, we review the relevant definitions, terminology, and basic results about $s$-core partitions and the $s$-core operation on partitions, mostly from Fayers \cite{F1,Fayers} and Johnson \cite{Johnson}. Section \ref{SEC:t-core-of-s-core-exposition} provides the fundamental results on $t$-cores of $s$-cores, giving the background needed to state Theorems \ref{THM:general-Fayers} and \ref{THM:self-conjugate-Fayers} (Fayers' conjectures). We isolate our main new observation as Proposition \ref{PROP:s-set-t-set-interaction}, which first gives a cleaner proof of a key proposition from \cite{F1}, and later features in our stabilizer computations.

Section \ref{SEC:computation-inputs} describes the relevant computations for Section \ref{SEC:main-theorem-proofs} (on general cores): we compute the sizes of the stabilizers appearing in Theorem \ref{THM:general-Fayers} and express $s$-set data, stabilizer sizes, and partition sizes in Johnson's \emph{$z$-coordinates}. We also extend the $z$-coordinates, and review the cyclic shifts used to compute $z$-coordinate sums of cyclic functions, such as the stabilizer and partition sizes.

Section \ref{SEC:main-theorem-proofs} presents the main results of the paper, namely explicit proofs of Theorems \ref{THM:general-Armstrong} and \ref{THM:general-Fayers} (general conjectures). Section \ref{SEC:self-conjugate-analogs} presents self-conjugate analogs of the general analysis, building up to explicit proofs of Theorems \ref{THM:self-conjugate-Armstrong} and \ref{THM:self-conjugate-Fayers} (self-conjugate conjectures).

Section \ref{SEC:Amdeberhan--Leven--conjecture} uses cyclic shifts in the extended $z$-coordinates from Section \ref{SEC:z-coordinates-intro} to quickly prove Theorem \ref{THM:Amdeberhan--Leven--conjecture}. Section \ref{SEC:future-work} discusses possibilities for future work, including the question of calculating sums of higher powers (or moments) of the $(s,t)$-core partition sizes.

\section{Background: \texorpdfstring{$s$}{s}-core partitions and operation}
\label{SEC:s-core-exposition}

Experts can quickly skim this section for the notation used in our paper (as the literature seems to have many different conventions), particularly the framework of beta-sets (Section \ref{SEC:beta-sets}) for studying hooks, and two parameterizations of $s$-cores: Johnson's \emph{$c$-coordinates} (Section \ref{SEC:charge-intro}) from \cite{Johnson}, and Fayers' $s$-sets and \emph{$a$-coordinates} (Section \ref{SEC:s-sets-and-a-coordinates-vs-c-coordinates}) from \cite{F1,Fayers}.\begin{arxiv}

Recall the following definition from the introduction.

\begin{sidedef}
\label{REF:SIDE:s-core-definitions}

Fix $s\ge1$. As in \cite{Fayers}, let $\mcal{C}_s$ denote the set of \emph{$s$-cores} $\lambda$, i.e. partitions with no rim $s$-hooks (or equivalently, no traditional hooks of length $s$). For any $s,t\ge1$, we often call $\mcal{C}_s\cap \mcal{C}_t$ the set of \emph{(simultaneous) $(s,t)$-cores}.
\end{sidedef}

\begin{sidedef}
\label{REF:SIDE:s-core-operation-definition}

For an arbitrary partition $\lambda$, say an $s$-core partition $\mu$ is an \emph{$s$-core of $\lambda$} if it can be obtained (starting from $\lambda$ itself) by repeatedly removing rim $s$-hooks. (Clearly $\lambda$ is an $s$-core if and only if $\lambda$ is an $s$-core of $\lambda$ itself.)
\end{sidedef}

\begin{sidermk}
As remarked in the introduction, there is in fact a unique such $s$-core $\mu$ (of $\lambda$) by Proposition \ref{REF:s-core-operation-is-well-defined}, which we denote by $\lambda^s$. So $\lambda$ is an $s$-core if and only if $\lambda = \lambda^s$.
\end{sidermk}
\end{arxiv}

\subsection{Beta-sets}
\label{SEC:beta-sets}

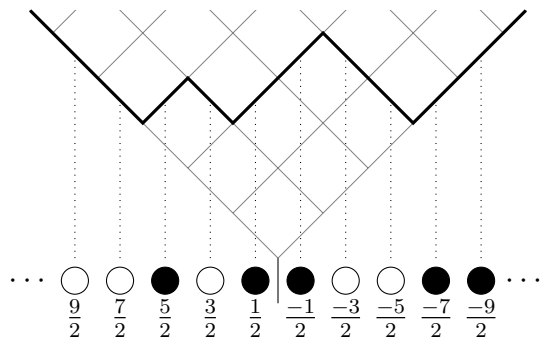
\begin{figure}
\centering

\trimbox{0cm 1.7cm 0cm 0cm}{
\begin{tikzpicture}
\begin{scope}[gray, very thin, scale=.6]
\clip (-5.5, 5.5) rectangle (5.5, -5);
\draw[rotate=45, scale=1.412] (0,0) grid (6,6);
\end{scope}

\begin{scope}[rotate=45, very thick, scale=.6*1.412]
\draw (0,5.5) -- (0, 3) -- (1,3) -- (1,2) -- (3,2) -- (3,0) -- (5.5,0);
\end{scope}

\begin{scope}[scale=.6, dotted]

\draw (-4.5,0) -- (-4.5, 4.5);
\draw (-3.5,0) -- (-3.5, 3.5);
\draw (-2.5,0) -- (-2.5, 3.5);
\draw (-1.5,0) -- (-1.5, 3.5);
\draw (-.5,0) -- (-.5, 3.5);
\draw (.5,0) -- (.5, 4.5);
\draw (1.5,0) -- (1.5, 4.5);
\draw (2.5,0) -- (2.5, 3.5);
\draw (3.5,0) -- (3.5, 3.5);
\draw (4.5,0) -- (4.5, 4.5);
\end{scope}

\begin{scope}[scale=.6, yshift=-.5cm]
\draw[solid] (0,-.5) -- (0,.5);
\draw (-5.5,0) node{$\cdots$};
\draw (-4.5,0) circle (.3) node[below=3pt]{$\frac{9}{2}$};
\draw (-3.5,0) circle (.3) node[below=3pt]{$\frac{7}{2}$};
\filldraw (-2.5,0) circle (.3) node[below=3pt]{$\frac{5}{2}$};
\draw (-1.5,0) circle (.3) node[below=3pt]{$\frac{3}{2}$};
\filldraw (-.5,0) circle (.3) node[below=3pt]{$\frac{1}{2}$};
\filldraw (.5,0) circle (.3) node[below=3pt]{$\frac{-1}{2}$};
\draw (1.5,0) circle (.3) node[below=3pt]{$\frac{-3}{2}$};
\draw (2.5,0) circle (.3) node[below=3pt]{$\frac{-5}{2}$};
\filldraw (3.5,0) circle (.3) node[below=3pt]{$\frac{-7}{2}$};
\filldraw (4.5,0) circle (.3) node[below=3pt]{$\frac{-9}{2}$};
\draw (5.5,0) node{$\cdots$};
\end{scope}
\end{tikzpicture}
}

\caption{\emph{Russian notation} for $\lambda = (3,2,2,0,\ldots)$, from Johnson's arXiv source \cite{Johnson}. (As pointed out by an anonymous expert, the idea itself has been used earlier by Okounkov and Reshetikhin.) The infinite \emph{rim boundary} is thickened. A \emph{rim edge} has a filled circle below if and only if the edge \emph{slopes upwards}. By convention, our labeling increases from \emph{right to left} (opposite the usual Cartesian $x$-axis).} \label{FIG:johnson-russian-notation}
\end{figure}

To each partition one naturally associates a \emph{beta-set} illuminating the hook length structure. It is often easier to work with beta-sets than with partitions themselves.

\begin{definition}[{\cite[Section 2.2]{Johnson}}]
\label{REF:beta-value-set-of-partition}

For any (infinite weakly decreasing) partition $\lambda = (\lambda_1,\lambda_2,\ldots)$, let $B^\lambda\subset\ZZ$ denote the (infinite) \emph{beta-set} of \emph{beta-values}, i.e. integers $x$ such that the the \emph{rim edge} (with midpoint) at position $x+\frac12$ \emph{slopes upwards} in Figure \ref{FIG:johnson-russian-notation}.
\end{definition}

\begin{remark}
Most sources (e.g. \cite{F1,Fayers,James-expos}) instead define $B^\lambda \colonequals \{\lambda_i - i\}_{i\ge1}$. The definitions are equivalent by a simple coordinate geometry argument. Some sources, notably Anderson \cite{Anderson} (and the poset method users \cite{SZ,Aggarwal,AL,Amol}, by extension) and Vandehey \cite{Vandehey}, instead distinguish the \emph{finite} set $(B^\lambda + r)\cap \ZZ_{>0} = \{(\lambda_1 - 1) + r,\ldots,(\lambda_r - r) + r\}$ of \emph{first column hook lengths}, where $r$ denotes the \emph{length} of $\lambda$, defined so that $\lambda_1 \ge \cdots \ge \lambda_r > 0 = \lambda_{r+1} = \cdots$. This perspective works particularly well for the study of \emph{maximal cores} (see e.g. \cite{Vandehey} for $(s,t)$-cores and \cite{Amol} for $(a,b,c)$-cores).
\end{remark}

\begin{proposition}[{Robinson \cite[2.8]{Robinson-expos}}; {Johnson \cite[Section 2.3.1]{Johnson}}]
\label{REF:beta-interpret-hooks}

A rim $s$-hook is parameterized (uniquely) by an element $x\in B^\lambda \setminus (B^\lambda + s)$. In particular, a partition $\lambda$ is an $s$-core if and only if $B^\lambda - s \subseteq B^\lambda$.
\end{proposition}

\begin{markednewpar}\leavevmode\end{markednewpar}\begin{arxiv}
\begin{proof}
In Figure \ref{FIG:johnson-russian-notation}, a rim $s$-hook is parameterized by an upwards-sloping rim edge at position $x+\frac12$ (i.e. $x\in B^\lambda$) followed by a downwards-sloping rim edge at position $x-s+\frac12$ (i.e. $x-s\notin B^\lambda$). This is equivalent to $x\in B^\lambda \setminus (B^\lambda + s)$.
\end{proof}

\begin{sidermk}
The $s$-core criterion $B^\lambda - s \subseteq B^\lambda$ is due to Robinson \cite[2.8]{Robinson-expos} (according to \cite{Fayers}). As mentioned in the introduction, it shows that $\lambda\in \mcal{C}_s$ if and only if $\lambda$ has no hook of length \emph{divisible} by $s$, establishing an alternative common definition of $\mcal{C}_s$.
\end{sidermk}

\begin{sidermk}

Fix coprime $s,t\ge1$. Then via beta-sets (more precisely, after negation and suitable translation), the \emph{numerical semigroups} (subsets of $\ZZ_{\ge0}$ that contain $0$, are closed under addition, and contain all sufficiently large integers) containing $s,t$ inject into the set of $(s,t)$-cores---as mentioned in the introduction and the future work sections.
\end{sidermk}
\end{arxiv}

\subsection{Johnson's signed `charge' measure \texorpdfstring{$c$}{c}}
\label{SEC:charge-intro}

In light of Proposition \ref{REF:beta-interpret-hooks}, one would like to have a clean description of possible beta-sets. The basic tool for this is \emph{charge}, which also parameterizes $s$-cores in Corollary \ref{REF:c-coordinates-parameterization-of-s-cores}.

\begin{definition}[c.f. {\cite[Definition 2.1 and Section 2.2]{Johnson}}; \cite{Cranks-t-cores-expos}; \cite{Physics-expos}]
\label{REF:c-coordinates-charge-definition}

Call a set $S\subseteq\ZZ$ \emph{good} if $S\cap\ZZ_{>0}$ and $\ZZ_{<0}\setminus S$ are both finite. For good $S$, define the signed \emph{$s$-charge} measure
\[
c_{s,i}(S)
\colonequals
\#[(-1-i+s\ZZ)\cap(\ZZ\setminus S)\cap\{x+\tfrac12<0\}]
-
\#[(-1-i+s\ZZ)\cap S\cap\{x+\tfrac12>0\}]
\]
for any $s\ge1$ and $i\in\ZZ/s\ZZ$. Then $\sum_{i\in\ZZ/s\ZZ} c_{s,i}(S) = c_{1,0}(S)$ is the \emph{total charge} of $S$.

For any partition $\lambda$, we may define $c_{s,i}(\lambda) \colonequals c_{s,i}(B^\lambda)$, since $B^\lambda$ is good.
\end{definition}

\begin{remark}
We will not use the notions of \emph{electron}, \emph{positron}, and \emph{Maya diagram} from Johnson's exposition.
\end{remark}

The basic importance of charge is given by the following charge condition.

\begin{proposition}[{\cite[Section 2.2]{Johnson}}]
\label{REF:c-coordinates-charge-sum-condition}

Fix $s\ge1$. Via beta-sets, partitions are parameterized (uniquely) by good sets $S$ (as defined in Definition \ref{REF:c-coordinates-charge-definition}) with $\sum_{i\in\ZZ/s\ZZ} c_{s,i}(S)$ equal to $0$.
\end{proposition}

Combining the $s$-core criterion $B^\lambda - s \subseteq B^\lambda$ (from Proposition \ref{REF:beta-interpret-hooks}) with the preceding zero charge invariant yields the following simple \emph{$c$-coordinates parameterization} of $s$-cores. For the conversion to Fayers' \emph{$a$-coordinates}, see Section \ref{SEC:s-sets-and-a-coordinates-vs-c-coordinates}.

\begin{corollary}[{Johnson \cite[Lemma 2.8]{Johnson}; Garvan--Kim--Stanton \cite[Bijection 2]{Cranks-t-cores-expos}}]
\label{REF:c-coordinates-parameterization-of-s-cores}

Fix $s\ge1$. The $s$-cores are parameterized (uniquely) by $s$-tuples $(c_{s,i})_{i\in\ZZ/s\ZZ}\in \ZZ^s$ summing to $0$.
\end{corollary}

\subsection{The \texorpdfstring{$s$}{s}-core operation is well-defined}

Beta-sets not only parameterize rim $s$-hooks (Proposition \ref{REF:beta-interpret-hooks}), but also conveniently describe the \emph{removal} of rim $s$-hooks, as follows.

\begin{proposition}[{\cite[Section 2.3.1]{Johnson}}]
\label{REF:beta-interpret-hook-removal}

The removal (from a partition $\lambda$) of a rim $s$-hook, say parameterized by $x\in B^\lambda \setminus (B^\lambda + s)$, corresponds to an \emph{$s$-push}: replacing the element $x$ of $B^\lambda$ by $x-s$. The process also \emph{preserves the charge $s$-tuple} $(c_{s,i})_{i\in\ZZ/s\ZZ}$.
\end{proposition}

\begin{markednewpar}\leavevmode\end{markednewpar}\begin{arxiv}
\begin{proof}
In Figure \ref{FIG:johnson-russian-notation}, a simple geometric argument shows that removing the rim $s$-hook parameterized by $x\in B^\lambda \setminus (B^\lambda + s)$ simply \emph{swaps} the slopes (\emph{up} or \emph{down}) of the \emph{rim edges} at positions $x+\frac12$ and $x-s+\frac12$, while preserving the slopes at all other positions. This corresponds to replacing $x\in B^\lambda$ by $x-s$, and leaving the other elements alone.

Furthermore, replacing $x\in B^\lambda\setminus(B^\lambda + s)$ by $x-s$ preserves the charge $s$-tuple $(c_{s,i})_{i\in\ZZ/s\ZZ}$; by Definition \ref{REF:c-coordinates-charge-definition}, there is nothing to check when $-1-i \not\equiv x\pmod{s}$. One can then check that $c_{s,-1-x}$ stays constant.
\end{proof}

\begin{sidermk}

In the \emph{infinite $s$-abacus} terminology in the literature---see e.g. \cite{F1,Johnson}---the \emph{$s$-push} is described as ``pushing the bead at $x$ by $s$ to the unfilled position $x-s$''; or in the electron diagram (see Figure \ref{FIG:johnson-russian-notation}) of \cite{Johnson} as ``inverting the filled energy state $x+\frac12$ with the unfilled energy state $x-s+\frac12$'' (the authors of \cite{Physics-expos} would likely also describe the operation this way).
\end{sidermk}
\end{arxiv}

\begin{definition}[c.f. {\cite[Section 2.3. Abaci.]{Johnson}}]
\label{REF:s-push-of-beta-set}

The \emph{$s$-push of a good set $S$} is the well-defined result of repeatedly applying the \emph{$s$-pushes} defined in Proposition \ref{REF:beta-interpret-hook-removal}. Explicitly, the $s$-push of $S$ can be described as follows:
\begin{itemize}
\item Fix a residue class $i+s\ZZ$; then $S \cap \{i+s\ZZ\}$ takes the form $\{\ldots,x-2s,x-s,x,x+\alpha_1 s,\ldots,x+\alpha_k s\}$ for any sufficiently small $x\in S\cap\{i+s\ZZ\}$.

\item Then the $s$-push of $S$, restricted to the residue class $i+s\ZZ$, is $\{\ldots,x-2s,x-s,x,x+s,\ldots,x+ks\}$.
\end{itemize}

\end{definition}

Repeatedly applying Proposition \ref{REF:beta-interpret-hook-removal} shows that the $s$-core operation is well-defined.

\begin{proposition}[c.f. {\cite[Section 2.3.1]{Johnson}}]
\label{REF:s-core-operation-is-well-defined}

Every partition $\lambda$ has a unique \emph{$s$-core}\begin{arxiv} (under Side Definition \ref{REF:SIDE:s-core-operation-definition})\end{arxiv}, denoted by $\lambda^s \in \mcal{C}_s$. Its beta-set $B^{\lambda^s}$ is the \emph{$s$-push} of the original beta-set $B^\lambda$, so $c_{s,i}(\lambda) = c_{s,i}(\lambda^s)$ for all $i\in\ZZ/s\ZZ$.
\end{proposition}

\begin{markednewpar}\leavevmode\end{markednewpar}\begin{arxiv}
\begin{sidermk}

As explained in \cite[Bijection 1]{Cranks-t-cores-expos}, the $s$-core operation leads to a nontrivial formula for the generating function for $s$-cores, indexed by size. The authors also study the generating function for self-conjugate $s$-cores.
\end{sidermk}
\end{arxiv}

\subsection{Fayers' \texorpdfstring{$s$}{s}-sets versus Johnson's \texorpdfstring{$c$}{c}-coordinates}
\label{SEC:s-sets-and-a-coordinates-vs-c-coordinates}

For any $s$-core $\lambda$, let $a_{s,i}(\lambda) \colonequals s+ \max [B^{\lambda}\cap (i + s\ZZ)]$ for each $i\in\ZZ$. But Definition \ref{REF:c-coordinates-charge-definition} and Corollary \ref{REF:c-coordinates-parameterization-of-s-cores} give $c_{s,-1-i}(\lambda)$ in terms of $B^{\lambda}\cap(i+s\ZZ)$. Comparing the two descriptions gives $a_{s,i} = i - sc_{s,-1-i}$ for $0\le i\le s-1$. In particular, $\sum_{i\in\ZZ/s\ZZ} c_{s,i} = 0$ (Proposition \ref{REF:c-coordinates-charge-sum-condition}) is equivalent to $\sum_{i\in\ZZ/s\ZZ} a_{s,i} = \binom{s}{2}$.

\begin{proposition}[Fayers \cite{F1}; {\cite[Section 3.3]{Fayers}}]
\label{REF:s-sets}

Fix $s\ge1$. The $s$-cores are parameterized (uniquely) by \emph{$s$-sets} $\mcal{S}_s = \{a_{s,i}\}_{i\in\ZZ/s\ZZ}$ of \emph{$a$-coordinates} summing to $\binom{s}{2}$ with $a_{s,i} \equiv i\pmod{s}$ for all $i\in\ZZ$. Explicitly, $\mcal{S}_s(\lambda) \colonequals (B^{\lambda}+s)\setminus B^{\lambda}$ for $\lambda\in \mcal{C}_s$.
\end{proposition}

\begin{remark}
In view of $c_{s,i}(\lambda) = c_{s,i}(\lambda^s)$ from Proposition \ref{REF:s-core-operation-is-well-defined}, it would be meaningful to define $a_{s,i}(\lambda) \colonequals a_{s,i}(\lambda^s)$ for \emph{any} $\lambda$. However, we will only speak of $s$-sets and $a$-coordinates of $s$-cores. Charge will suffice for our greater needs in Proposition \ref{PROP:conjugate-charge} and Corollary \ref{COR:conjugation-commutes-with-s-core-operation}.
\end{remark}

\section{Background: \texorpdfstring{$t$}{t}-cores of \texorpdfstring{$s$}{s}-cores}
\label{SEC:t-core-of-s-core-exposition}

In this section, we go through the fundamental results on $t$-cores of $s$-cores. In particular, we isolate the crucial Proposition \ref{PROP:s-set-t-set-interaction}, meanwhile giving a cleaner proof of a key result of Fayers (see Proposition \ref{PROP:faithful-invariant}). This motivates Fayers' `level $t$ action on $s$-cores' (Definition \ref{DEF:level-t-group-action}), defining the stabilizers in Theorem \ref{THM:general-Fayers} (Fayers' general conjecture).

\subsection{When is a \texorpdfstring{$t$}{t}-core an \texorpdfstring{$s$}{s}-core?}

The following simple criterion parameterizes $\mcal{C}_s\cap\mcal{C}_t$ in $a_t$-coordinates. It implicitly appears throughout this paper and elsewhere.

\begin{lemma}[{Johnson \cite[Lemma 3.1]{Johnson}; re-worded by Fayers \cite[Lemma 3.8]{Fayers}}]
\label{LEM:c,a-coordinate-parameterization-s,t-cores}

In Fayers' $a$-coordinates (see Proposition \ref{REF:s-sets}), the set $\mcal{C}_s\cap\mcal{C}_t$ of $s$-cores \emph{within the affine lattice $\mcal{C}_t$ of $t$-cores} is defined by the system of inequalities $a_{t,i} \ge a_{t,i+s} - s$ for $i\in \ZZ/t\ZZ$.
\end{lemma}

\begin{markednewpar}\leavevmode\end{markednewpar}\begin{arxiv}
\begin{proof}
The proof is the same as the first half of the proof of Proposition \ref{PROP:s-set-t-set-interaction}. Suppose $\lambda\in \mcal{C}_t$, i.e. the $t$-core criterion $B^\lambda - t \subseteq B^\lambda$ holds (from Proposition \ref{REF:beta-interpret-hooks}). Then by definition of $a_{t,i}\equiv i\pmod{t}$ and $a_{t,i+s}\equiv i+s\pmod{t}$, we have
\begin{align*}
B^\lambda\cap (i+t\ZZ) &= \{\ldots,a_{t,i}-2t,a_{t,i}-t\} \\
(B^\lambda - s)\cap (i+t\ZZ) &= \{\ldots,[a_{t,i+s}-s]-2t,[a_{t,i+s}-s]-t\},
\end{align*}
for any $i\in\ZZ$. Thus the $s$-core criterion $B^\lambda - s\subseteq B^\lambda$ holds if and only if the inequalities $a_{t,i+s} - s \le a_{t,i}$ hold for all $i\in\ZZ/t\ZZ$.
\end{proof}
\end{arxiv}

\subsection{Comparing \texorpdfstring{$t$}{t}-cores of different \texorpdfstring{$s$}{s}-cores}

Propositions \ref{PROP:extending-Olsson-theorem}, \ref{PROP:s-set-t-set-interaction}, and \ref{PROP:faithful-invariant} below are the key conceptual inputs for comparing $t$-cores of various $s$-cores. We start by extending Olsson's theorem \cite{Olsson}, following Fayers \cite{F1}.

\begin{proposition}[c.f. {\cite[proof of Proposition 4.1]{F1}}]
\label{PROP:extending-Olsson-theorem}
Fix \emph{any} $s,t\ge1$ and $\lambda\in \mcal{C}_s$. Then $\lambda^t\in \mcal{C}_s\cap \mcal{C}_t$, and furthermore $\mcal{S}_s(\lambda^t) \equiv \mcal{S}_s(\lambda) \pmod{t}$ (viewed as multisets of residues).
\end{proposition}

\begin{proof}[Proof sketch]

Given $\lambda\in \mcal{C}_s$, Fayers inductively constructs a sequence of \emph{$s$-core partitions} $\lambda^{(0)},\ldots,\lambda^{(-m)}$ (with $m\ge0$) such that $\lambda^{(0)} = \lambda$, the term $\lambda^{(-i-1)}$ is an $s$-core obtained from $\lambda^{(-i)}$ by removing a certain sequence of rim $t$-hooks, and $\lambda^{(-m)}\in \mcal{C}_s\cap \mcal{C}_t$. He shows, under this construction, that $\mcal{S}_s(\lambda^{(-i-1)}) \equiv \mcal{S}_s(\lambda^{(-i)}) \pmod{t}$ for $0\le i < m$, so $\mcal{S}_s(\lambda^{(-m)}) \equiv \mcal{S}_s(\lambda)\pmod{t}$. But $\lambda^{(-m)} = \lambda^t$ by uniqueness of the $t$-core (Proposition \ref{REF:s-core-operation-is-well-defined}).
\end{proof}

\begin{remark}
One can avoid induction by first comparing $(B^{\lambda} - s)\cap (j+t\ZZ)$ with $B^{\lambda}\cap (j+t\ZZ)$ as $j\in\ZZ$ varies, and then (in view of Proposition \ref{REF:s-core-operation-is-well-defined}) the \emph{$t$-pushes} $(B^{\lambda^t} - s)\cap (j+t\ZZ)$ and $B^{\lambda^t}\cap (j+t\ZZ)$.
\end{remark}

\begin{markednewpar}\leavevmode\end{markednewpar}\begin{arxiv}
\begin{proof}[Alternative proof avoiding induction]
Suppose $\lambda\in \mcal{C}_s$, i.e. the $s$-core criterion $B^\lambda - s\subseteq B^\lambda$ holds (from Proposition \ref{REF:beta-interpret-hooks}), and fix an integer $j$. Since $\ZZ_{<0}\setminus B^\lambda$ is finite, we have
\begin{align*}
(B^\lambda - s)\cap (j+t\ZZ) &= \{\ldots,u_j-2t,u_j-t,u_j,u_j+\alpha_1 t,\ldots,u_j+\alpha_k t\} \\
B^\lambda\cap (j+t\ZZ) &= \{\ldots,u_j-2t,u_j-t,u_j,u_j+\beta_1 t,\ldots,u_j + \beta_\ell t\},
\end{align*}
for any choice of a sufficiently small element $u_j\equiv j\pmod{t}$ of $(B^\lambda-s)\cap(j+t\ZZ)$ (i.e. so that $B^\lambda - s$ contains $u_j - t, u_j - 2t, u_j - 3t,\ldots$), and $\alpha_1,\ldots,\alpha_k$ and $\beta_1,\ldots,\beta_\ell$ are increasing (possibly empty) sequences of \emph{positive} integers with $\{\alpha_1,\ldots,\alpha_k\} \subseteq \{\beta_1,\ldots,\beta_\ell\}$ (so $0\le k\le \ell$). Thus $\mcal{S}_s(\lambda)-s$, i.e. $B^\lambda \setminus (B^\lambda - s)$ by Proposition \ref{REF:s-sets}, contains $\ell - k$ residues congruent to $j\pmod{t}$.

But $B^{\lambda^t}$ is the \emph{$t$-push} of $B^\lambda$ by Proposition \ref{REF:s-core-operation-is-well-defined}, so the previous computations translate to
\begin{align*}
(B^{\lambda^t} - s)\cap (j+t\ZZ) &= \{\ldots,u_j-2t,u_j-t,u_j,u_j+t,\ldots,u_j+kt\} \\
B^{\lambda^t}\cap (j+t\ZZ) &= \{\ldots,u_j-2t,u_j-t,u_j,u_j+t,\ldots,u_j + \ell t\}.
\end{align*}
Varying over all $j$ shows that $B^{\lambda^t}-s \subseteq B^{\lambda^t}$ (as $k\le \ell$ always), so the $t$-core $\lambda^t$ is indeed still an $s$-core, and furthermore, $\mcal{S}_s(\lambda^t)-s = B^{\lambda^t} \setminus (B^{\lambda^t} - s)$ (valid for the $s$-core $\lambda^t$) contains $\ell - k$ residues congruent to $j\pmod{t}$.

Finally, varying over all $j$ establishes $\mcal{S}_s(\lambda) - s \equiv \mcal{S}_s(\lambda^t) - s\pmod{t}$.
\end{proof}
\end{arxiv}

We isolate our main new observation as the following proposition, which first gives a cleaner proof of a key result from \cite{F1} (see Proposition \ref{PROP:faithful-invariant}), and later features in our stabilizer computations, namely Propositions \ref{PROP:stab-computation} (general case) and \ref{PROP:stab-computation-self-conjugate} (self-conjugate analog).

\begin{proposition}
\label{PROP:s-set-t-set-interaction}

Fix \emph{any} $s,t\ge1$ and $\lambda\in \mcal{C}_s$. Then \mathd{\card{[\mcal{S}_s(\lambda)-s]\cap (j+t\ZZ)} = \frac{1}{t}(a_{t,j}(\lambda^t) - [a_{t,j+s}(\lambda^t) - s]).}
\end{proposition}

\begin{proof}
Proposition \ref{PROP:extending-Olsson-theorem} says $\lambda^t\in \mcal{C}_s$ and $\mcal{S}_s(\lambda) \equiv \mcal{S}_s(\lambda^t)\pmod{t}$. Thus it suffices to prove the result with $\lambda$ replaced by its $t$-core $\lambda^t$.

In other words, we may \Wlog{} assume $\lambda\in \mcal{C}_s\cap\mcal{C}_t$. Recall the $t$-core criterion $B^\lambda - t\subseteq B^\lambda$ from Proposition \ref{REF:beta-interpret-hooks}. Then by definition of $a_{t,j}\equiv j\pmod{t}$ and $a_{t,j+s}\equiv j+s\pmod{t}$, we have
\begin{align*}
B^\lambda\cap (j+t\ZZ) &= \{\ldots,a_{t,j}-2t,a_{t,j}-t\} \\
(B^\lambda - s)\cap (j+t\ZZ) &= \{\ldots,[a_{t,j+s}-s]-2t,[a_{t,j+s}-s]-t\}.
\end{align*}
However, Proposition \ref{REF:s-sets} gives $\mcal{S}_s(\lambda)-s = B^\lambda \setminus (B^\lambda - s)$ for the $s$-core $\lambda$ (with the criterion $B^\lambda - s\subseteq B^\lambda$ implicit), so $[\mcal{S}_s(\lambda)-s] \cap (j+t\ZZ)$ is the difference-$t$ arithmetic progression $\{a_{t,j+s}-s,\ldots,a_{t,j}-t\}$ of nonnegative length $\frac1t (a_{t,j} - [a_{t,j+s} - s])$.
\end{proof}

\begin{remark}
\label{RMK:key-quantities-s-set-t-set-interaction}

Compare with both the statement \emph{and} proof of Lemma \ref{LEM:c,a-coordinate-parameterization-s,t-cores},
which can be rephrased in terms of the quantities $\omega \colonequals \frac1t (a_{t,j} - [a_{t,j+s} - s])$.
When $s,t$ are coprime, these re-index in Corollary \ref{COR:s-set-and-stabilizer-data-in-z-coordinates}
to form Johnson's \emph{$z$-coordinates} for $(s,t)$-cores $\lambda$ (as vaguely mentioned in the introduction).

Furthermore, for \emph{arbitrary} $t$-cores $\lambda$, the quantity $\omega$ is meaningful not only when $\omega\ge0$ (in which case $\omega = \card{(B^\lambda\setminus(B^\lambda-s))\cap (j+t\ZZ)}$), but also when $\omega\le0$ (in which case $-\omega\ge0$ counts the size of $((B^\lambda - s)\setminus B^\lambda)\cap (j+t\ZZ)$, i.e. the number of rim $s$-hooks coming---via Proposition \ref{REF:beta-interpret-hooks}---from beta-values $x+s\in B^\lambda\cap(j+s+t\ZZ)$ with $(x+s) - s\notin B^\lambda$). The latter observation is essentially due to Paul Johnson (via private correspondence), and gives combinatorial significance to our modest extension of his $z$-coordinates (see the second halves of Propositions \ref{PROP:change-to-z-coordinates} and \ref{PROP:change-to-z-coordinates-self-conjugate}).

\end{remark}

Proposition \ref{PROP:s-set-t-set-interaction} cleanly proves a result of Fayers ``crucial'' for ``comparing the $t$-cores of different $s$-cores'' \cite{F1}.

\begin{proposition}[Extension of {\cite[Proposition 4.1]{F1}}]
\label{PROP:faithful-invariant}
Fix coprime $s,t\ge1$. Then the $t$-core $\lambda^t$ of an $s$-core $\lambda$ is uniquely determined by the multiset of modulo $t$ residues $\mcal{S}_s(\lambda)\pmod{t}$. Combined with Proposition \ref{PROP:extending-Olsson-theorem}, we conclude that $\lambda,\mu\in \mcal{C}_s$ have the same $t$-core if and only if the multisets of modulo $t$ residues $\mcal{S}_s(\lambda),\mcal{S}_s(\mu)\pmod{t}$ are congruent.
\end{proposition}

\begin{proof}

Fix $\lambda\in \mcal{C}_s$, so Proposition \ref{PROP:s-set-t-set-interaction} gives $\card{\mcal{S}_s(\lambda)\cap (j+s+t\ZZ)} = \frac1t (a_{t,j}(\lambda^t) - [a_{t,j+s}(\lambda^t) - s])$. Since $s,t$ are coprime, it follows that $\mcal{S}_s(\lambda)\pmod{t}$ determines $\mcal{S}_t(\lambda^t) = \{a_{t,j}(\lambda^t)\}$ up to translation. The sum condition $\sum_{j\in\ZZ/t\ZZ} a_{t,j}(\lambda^t) = \binom{t}{2}$ (from Proposition \ref{REF:s-sets}) then singles out a unique translate equal to $\mcal{S}_t(\lambda^t)$, which corresponds under Proposition \ref{REF:s-sets} to a unique $t$-core $\lambda^t$.
\end{proof}

\subsection{Level \texorpdfstring{$t$}{t} action and statement of Fayers' general conjecture}

Proposition \ref{PROP:faithful-invariant} motivates the following group action on $\mcal{C}_s$, for which Corollary \ref{COR:level-t-group-action-orbits} will hold almost by definition.

\begin{definition}[c.f. Fayers {\cite[Section 3.2]{F1}; \cite[Section 3.1]{Fayers}}]
\label{DEF:level-t-group-action}
Fix coprime $s,t\ge1$. Let $G_{s,t}$ be the set of permutations $f\colon\ZZ\to \ZZ$ such that
\begin{itemize}
\item $f$ is \emph{$s$-periodic}, i.e. $f(m+s) = f(m)+s$ for all $m$\begin{arxiv} (or equivalently, $f(m)-m = f(n)-n$ whenever $m+s\ZZ = n+s\ZZ$)\end{arxiv};

\item $f$ satisfies the \emph{sum-invariance condition} $\sum_{i=0}^{s-1} f(i) = \binom{s}{2}$\begin{arxiv} (or equivalently by $s$-periodicity, that $f$ preserves the sum of any set of representatives of the $s$ residue classes modulo $s$)\end{arxiv};

\item $f$ preserves residues modulo $t$, i.e. $f(m) \equiv m\pmod{t}$ for all $m$.
\end{itemize}

It is easy to check that $G_{s,t}$ has a group structure, and that it acts in the obvious way on the set of $s$-sets $\mcal{S}_s(\lambda)$ (of $s$-cores $\lambda$), or equivalently on beta-sets $B^\lambda$ (of $s$-cores $\lambda$). This induces an action on $\mcal{C}_s$, via Proposition \ref{REF:s-sets}.
\end{definition}

\begin{remark}[Different but equivalent definitions]
\label{RMK:difference-of-group-definitions}

We have not given Fayers' actual definition of the `level $t$ action of the $s$-affine symmetric group' (based on the $t=1$ case from \cite{Lascoux-expos}), but rather one equivalent by \cite[Proposition 3.5]{Fayers}, and easier to work with for our purposes.
\end{remark}

\begin{remark}[Natural action?]

It is not hard to show that a permutation $f\colon\ZZ\to\ZZ$ preserves the set of \emph{beta-sets} of $s$-cores if and only if it satisfies the explicit $s$-periodicity and sum-invariance conditions.\begin{arxiv} Preserving the set of $s$-sets alone is barely insufficient; for instance one may have anti-$s$-periodicity, i.e. $f(m+s) = f(m)-s$.\end{arxiv}
\end{remark}

In practice one often encounters group elements by their restrictions to $s$-sets.

\begin{proposition}[{\cite[Corollary 3.6]{Fayers}}]
\label{PROP:extending-bijections-of-s-sets-to-group-actions}

Let $\lambda,\mu$ be $s$-cores.
Suppose $\phi$ a set bijection $\phi\colon \mcal{S}_s(\lambda)\to \mcal{S}_s(\mu)$ that preserves residue classes modulo $t$.
Then $\phi$ uniquely extends to an element $f\in G_{s,t}$ (under Definition \ref{DEF:level-t-group-action}).
\end{proposition}

\begin{markednewpar}\leavevmode\end{markednewpar}\begin{arxiv}
\begin{proof}
The uniqueness is clear: the $s$-periodicity condition uniquely determines $f$ on the modulo $s$ residue classes $a_{s,i}(\lambda)+s\ZZ = i+s\ZZ$, hence on the whole set of integers. Explicitly, we have $f(m) = m + [f(a_{s,m}(\lambda)) - a_{s,m}(\lambda)] = m + [\phi(a_{s,m}(\lambda)) - a_{s,m}(\lambda)]$ for any integer $m$, as $a_{s,m} \equiv m\pmod{s}$ by definition.

It is then easy to check that this (uniquely determined) $f$ is actually a permutation of $\ZZ$ (it permutes the residue classes modulo $s$, because $\phi$ permutes $\mcal{S}_s(\lambda)$) preserving residues modulo $t$ (because $\phi$ preserves residues modulo $t$, we have $f(m) = m + [\phi(a_{s,m}(\lambda)) - a_{s,m}(\lambda)] \equiv m\pmod{t}$) and satisfying the sum condition
\[ \sum_{i=0}^{s-1}f(i) = \sum_{i=0}^{s-1} i + \sum_{i\in\ZZ/s\ZZ} [\phi(a_{s,i}(\lambda))-a_{s,i}(\lambda)] = \binom{s}{2} + \sum_{i\in\ZZ/s\ZZ} [a_{s,i}(\mu) - a_{s,i}(\lambda)] = \binom{s}{2}, \]
hence an element of $G_{s,t}$ (again, as defined in Definition \ref{DEF:level-t-group-action}).
\end{proof}
\end{arxiv}

\begin{corollary}[{\cite[Proposition 4.2 and Corollary 4.5]{F1}}]
\label{COR:level-t-group-action-orbits}

Fix coprime $s,t\ge1$. Then $\lambda,\mu\in \mcal{C}_s$ lie in the same $G_{s,t}$-orbit if and only if $\lambda^t = \mu^t$. In other words, each $G_{s,t}$-orbit of $\mcal{C}_s$ contains a unique $t$-core (hence an $(s,t)$-core), and any $\lambda\in \mcal{C}_s$ lies in $G_{s,t}\lambda^t$.
\end{corollary}

\begin{proof}[Sketch of more direct proof]

Fix $s$-cores $\lambda,\mu$. Proposition \ref{PROP:faithful-invariant} followed by Proposition \ref{PROP:extending-bijections-of-s-sets-to-group-actions} gives equivalence of $\lambda^t = \mu^t$ and $\mu \in G_{s,t}\lambda$. The re-phrasing follows by specializing to $\mu \colonequals \lambda^t$ (where $\lambda^t\in \mcal{C}_s$ follows from Proposition \ref{PROP:extending-Olsson-theorem}).
\end{proof}

\begin{markednewpar}\leavevmode\end{markednewpar}\begin{arxiv}
\begin{proof}[Full direct proof]
Let $\lambda,\mu$ be two $s$-cores. Proposition \ref{PROP:faithful-invariant} states that $\lambda^t = \mu^t$ if and only if we have a congruence of $s$-sets $\mcal{S}_s(\lambda) \equiv \mcal{S}_s(\mu)\pmod{t}$, i.e. there exists a bijection $\phi\colon \mcal{S}_s(\lambda) \to \mcal{S}_s(\mu)$ (of $s$-sets) preserving residues modulo $t$. By Proposition \ref{PROP:extending-bijections-of-s-sets-to-group-actions}, such bijections $\phi$ correspond to group elements $f\in G_{s,t}$ with restriction $f|_{\mcal{S}_s(\lambda)} = \phi$. So $\lambda^t = \mu^t$ if and only if $\mu = f\lambda$ for some $f\in G_{s,t}$, i.e. $\lambda,\mu$ lie in the same $G_{s,t}$-orbit.
\end{proof}
\end{arxiv}

Definition \ref{DEF:level-t-group-action} and Corollary \ref{COR:level-t-group-action-orbits} provide the background and context for Theorem \ref{THM:general-Fayers} (stated in the introduction).

\section{Key inputs for computation}
\label{SEC:computation-inputs}

In this section, we describe all the computational methods and results used to compute the sums in Theorems \ref{THM:general-Armstrong} and \ref{THM:general-Fayers}. First we compute the sizes of the stabilizers appearing in Theorem \ref{THM:general-Fayers}. Section \ref{SEC:z-coordinates-intro} compares Fayers' $a$-coordinates with a modest extension of Johnson's $z$-coordinates. In Section \ref{SEC:z-coordinates-partition-size} we give an explicit formula for the size of a $t$-core, and in Section \ref{SEC:z-coordinate-cyclic-shifts} we explain the standard cyclic shifts used to compute $z$-coordinate sums of cyclic functions, such as the stabilizer and partition sizes.

\subsection{Size of the stabilizer of an \texorpdfstring{$s$}{s}-core}
\label{SEC:stab-computation}

Most of the proof ideas for Theorem \ref{THM:general-Fayers} come from Johnson \cite{Johnson} and Fayers \cite{F1,Fayers}. The key new observation is the following computational simplification of Fayers' formula for the size of the stabilizer of an $s$-core, based on Proposition \ref{PROP:s-set-t-set-interaction}, which will simplify even further once we translate to Johnson's $z$-coordinates (see Corollary \ref{COR:s-set-and-stabilizer-data-in-z-coordinates}).

\begin{proposition}[c.f. {\cite[Proposition 3.7]{Fayers}}]
\label{PROP:stab-computation}

Fix coprime $s,t\ge1$, and $\lambda\in \mcal{C}_s$. Then with $G_{s,t}$ from Definition \ref{DEF:level-t-group-action}, $\card{\stab_{G_{s,t}}(\lambda)}$ equals $\prod_{j\in\ZZ/t\ZZ} [\frac{1}{t}(a_{t,j}(\lambda^t) - [a_{t,j+s}(\lambda^t) - s])]!$.
\end{proposition}


\begin{proof}
Using Proposition \ref{PROP:extending-bijections-of-s-sets-to-group-actions},
Fayers showed in \cite[Proposition 3.7]{Fayers} that $\stab_{G_{s,t}}(\lambda)$ (under a different but equivalent definition of $G_{s,t}$; see Remark \ref{RMK:difference-of-group-definitions})
has size
\mathd{
\prod_{j\in\ZZ/t\ZZ}\card{\mcal{S}_s(\lambda)\cap (j+t\ZZ)}!
= \prod_{j\in\ZZ/t\ZZ}\card{[\mcal{S}_s(\lambda)-s]\cap (j+t\ZZ)}!
.}
\begin{arxiv}
To prove this, take $f\in G_{s,t}$ (as defined in Definition \ref{DEF:level-t-group-action}); then in particular, $f$ preserves residue classes modulo $t$. By definition, $f$ lies in the stabilizer $\stab(\lambda)$ if and only if $f$ fixes the $s$-set $\mcal{S}_s(\lambda)$ (of the $s$-core $\lambda$), i.e. $f$ restricts to a permutation $\pi$ on the elements of $\mcal{S}_s(\lambda)$ also preserving residues modulo $t$. Observe that

\begin{itemize}
\item Any such permutation $\pi$ uniquely extends to an element $f\in G_{s,t}$. Indeed, this is just Proposition \ref{PROP:extending-bijections-of-s-sets-to-group-actions} applied to the bijection $\pi\colon \mcal{S}_s(\lambda) \mapsto \mcal{S}_s(\lambda)$ (of $s$-sets).

\item Such permutations $\pi$ of $\mcal{S}_s(\lambda)$ correspond to (disjoint) products of permutations of $\mcal{S}_s(\lambda)\cap (j+t\ZZ)$ (on the individual residue classes $j+t\ZZ$).
\end{itemize}
\end{arxiv}

Substituting in Proposition \ref{PROP:s-set-t-set-interaction} gives the result.
\end{proof}

\begin{markednewpar}\leavevmode\end{markednewpar}\begin{arxiv}
\begin{sidermk}
We will only explicitly use this result for $(s,t)$-cores $\lambda\in \mcal{C}_s\cap \mcal{C}_t$, when $s,t$ are coprime, in the proof of Fayers' general conjecture (Theorem \ref{THM:general-Fayers}). More precisely, we will use the $z$-coordinate translation given in Corollary \ref{COR:s-set-and-stabilizer-data-in-z-coordinates}.
\end{sidermk}
\end{arxiv}

\subsection{Johnson's \texorpdfstring{$z$}{z}-coordinates versus Fayers' \texorpdfstring{$t$}{t}-sets}
\label{SEC:z-coordinates-intro}

As reflected by the simple translation Corollary \ref{COR:s-set-and-stabilizer-data-in-z-coordinates}, Proposition \ref{PROP:stab-computation} and Remark \ref{RMK:key-quantities-s-set-t-set-interaction} provide one source of motivation for the following choice of coordinates. We not only review Johnson's parameterization of $(s,t)$-cores, but also extend it to arbitrary $t$-cores, given a parameter $s\ge1$ coprime to $t$.

\begin{proposition}[c.f. {\cite[Lemma 3.5]{Johnson}}: Johnson's $z$-coordinates versus Fayers' $t$-sets]
\label{PROP:change-to-z-coordinates}

Fix coprime $s,t\ge1$. The set $\mcal{C}_s\cap\mcal{C}_t$ of $(s,t)$-cores (viewed as $s$-cores within the set of $t$-cores) is parameterized by either of the sets $A_t(s),\map{TD}_t(s)$, described as follows.
\begin{itemize}
\item By Lemma \ref{LEM:c,a-coordinate-parameterization-s,t-cores}, the set of $t$-sets $\mcal{S}_t(\lambda)$ of $(s,t)$-cores $\lambda$, i.e. the set of points $A_t(s) = \{(a_{t,i})_{i\in\ZZ/t\ZZ}\}$ defined by the inequalities $a_{t,i} \ge a_{t,i+s} - s$, the sum condition $\sum_{i\in\ZZ/t\ZZ} a_{t,i} = \binom{t}{2}$, and congruence conditions $a_{t,i} \equiv i\pmod{t}$.

\item Johnson's \emph{trivial determinant representations} set $\map{TD}_t(s) = \{(z_{t,i})_{i\in\ZZ/t\ZZ}\}$ defined by the inequalities $z_{t,i}\ge0$, the sum condition $\sum_{i\in\ZZ/t\ZZ} z_{t,i} = s$, and congruence conditions $z_{t,i} \equiv 0\pmod{1}$ and $\sum_{i\in\ZZ/t\ZZ} i z_{t,i} \equiv 0\pmod{t}$.
\end{itemize}

An isomorphism (also preserving the ambient linear and simplex structures) is given by the invertible affine change of variables $z_{t,j} \colonequals \frac{1}{t}(a_{t,sj + k} - [a_{t,s(j+1) + k} - s])$, for $j\in\ZZ/t\ZZ$, where $k \colonequals \frac12 (s+1)(t-1) \in \ZZ$. The inverse map can be described by $a_{t,k+\ell s} - \frac{t-1}{2} =  \sum_{j=0}^{t-1} (\frac{t-1}{2} - j) z_{t,j+\ell}$, for $\ell\in\ZZ/t\ZZ$.


Under the same affine change of variables, the larger set $\mcal{C}_t$ of $t$-cores is parameterized by either of the following sets.
\begin{itemize}
\item By Proposition \ref{REF:s-sets}, the set of $t$-sets $\mcal{S}_t(\lambda)$ of $t$-cores $\lambda$, i.e. the set of points $\mcal{C}_t = \{(a_{t,i})_{i\in\ZZ/t\ZZ}\}$ defined by the sum condition $\sum_{i\in\ZZ/t\ZZ} a_{t,i} = \binom{t}{2}$, and congruence conditions $a_{t,i} \equiv i\pmod{t}$.

\item In $z$-coordinates, the set of points $\mcal{C}_t = \{(z_{t,i})_{i\in\ZZ/t\ZZ}\}$ defined by the sum condition $\sum_{i\in\ZZ/t\ZZ} z_{t,i} = s$, and congruence conditions $z_{t,i} \equiv 0\pmod{1}$ and $\sum_{i\in\ZZ/t\ZZ} i z_{t,i} \equiv 0\pmod{t}$.
\end{itemize}
\end{proposition}

\begin{markednewpar}\leavevmode\end{markednewpar}\begin{arxiv}
\begin{sidermk}
\label{RMK:SIDE:z-coordinates-t-core-clarification-of-s}

In the $z$-coordinate parameterization of the larger set $\mcal{C}_t$ (as opposed to $\mcal{C}_s\cap \mcal{C}_t$), one may think of $s$ as a ``purely algebraic parameter'' coprime to $t$, with applications to (for instance) the ``symmetric proof'' of Theorem \ref{THM:Amdeberhan--Leven--conjecture} given in Section \ref{SEC:Amdeberhan--Leven--conjecture}.
\end{sidermk}
\end{arxiv}

\begin{remark}
Although it will only matter for the asymmetric Theorem \ref{THM:general-Fayers}, not the symmetric Theorem \ref{THM:general-Armstrong}, we use, in Johnson's notation, the parameterization $\map{TD}_t(s)$, instead of $\map{TD}_s(t)$ as Johnson might for Armstrong's conjectures \cite{Johnson}.
\end{remark}

Before proving the result, we first translate $s$-set and stabilizer data to $z$-coordinates.

\begin{corollary}
\label{COR:s-set-and-stabilizer-data-in-z-coordinates}

Fix coprime $s,t\ge1$. Fix $\lambda\in \mcal{C}_s$. The $z$-coordinates parameterizing the $(s,t)$-core $\lambda^t$ are defined by the affine change of variables
\[ \underbrace{z_{t,j}(\lambda^t) \colonequals \tfrac{1}{t}(a_{t,sj+k}(\lambda^t) - [a_{t,s(j+1)+k}(\lambda^t) - s])}_\text{as defined in Proposition \ref{PROP:change-to-z-coordinates}} \underbrace{= \card{[\mcal{S}_s(\lambda) - s] \cap (sj+k+t\ZZ)}}_\text{by Proposition \ref{PROP:s-set-t-set-interaction}}, \]
where $k$ is the constant $\frac{1}{2}(s+1)(t-1)$. Furthermore, Proposition \ref{PROP:stab-computation} translates to
\begin{align*}
\card{\stab_{G_{s,t}}(\lambda)}
&= \prod_{j\in\ZZ/t\ZZ} \left( \frac{a_{t,j}(\lambda^t) - [a_{t,j+s}(\lambda^t) - s]}{t} \right)! \\
&\underbrace{= \prod_{j\in\ZZ/t\ZZ} \left( \frac{a_{t,sj+k}(\lambda^t) - [a_{t,s(j+1)+k}(\lambda^t) - s]}{t} \right)!}_\text{since $j+t\ZZ\mapsto sj+k+t\ZZ$ is bijective}
= \prod_{j\in\ZZ/t\ZZ} z_{t,j}(\lambda^t)!.
\end{align*}
\end{corollary}

\begin{markednewpar}\leavevmode\end{markednewpar}\begin{arxiv}
\begin{sidermk}
\label{RMK:SIDE:z-coordinates-significance}

Propositions \ref{PROP:stab-computation} and \ref{PROP:s-set-t-set-interaction} are not the only way to motivate the $z$-coordinates---Johnson considers them in \cite{Johnson} for aesthetic and practical reasons (i.e. repeatedly ``simplifying'' the parameterization of $(s,t)$-cores)---but the propositions do give the coordinates additional significance.
\end{sidermk}
\end{arxiv}

\begin{proof}[Proof of $a$-versus-$z$ isomorphism in Proposition \ref{PROP:change-to-z-coordinates}]

Let $\phi$ be the affine map sending a point $(x_i)_{i\in\ZZ/t\ZZ}$ on the $(t-1)$-dimensional plane $\sum_{i\in\ZZ/t\ZZ} X_i = \binom{t}{2}$ to $(y_j)_{j\in\ZZ/t\ZZ}$ with $y_j \colonequals \frac1t(x_{sj + k} - [x_{s(j+1) + k} - s])$. Then $\phi$ maps into the $(t-1)$-dimensional plane $\sum_{j\in\ZZ/t\ZZ} Y_j = s$. Since $s,t$ are coprime, it is easy to check that $\phi\colon \{\sum_{i\in\ZZ/t\ZZ} X_i = \binom{t}{2}\} \to \{\sum_{j\in\ZZ/t\ZZ} Y_j = s\}$ is injective, hence bijective.

Perhaps the most natural description of the inverse $\phi^{-1}$ is given (for $\ell\in\ZZ/t\ZZ$, noting that $\ell\mapsto s\ell+k$ is surjective modulo $t$) by evaluating $\sum_{j=0}^{t-1} (\frac{t-1}{2} - j) y_{j+\ell}\cdot t$, i.e. the sum $\sum_{j=0}^{t-1} (\frac{t-1}{2} - j)s + (\frac{t-1}{2} - j) x_{sj+s\ell+k} - (\frac{t-1}{2} - j) x_{s(j+1)+s\ell+k}$, which telescopes to
\[ \frac{t-1}{2}x_{s\cdot 0+s\ell+k} + \frac{t-1}{2}x_{s\cdot t+s\ell+k} - \sum_{j=0}^{t-2} x_{s(j+1)+s\ell+k} = tx_{s\ell+k} - \binom{t}{2}. \]
This yields the equality $\sum_{j=0}^{t-1} (\frac{t-1}{2} - j) y_{j+\ell} = x_{s\ell+k} - \frac{t-1}{2}$.

Having analyzed the ambient affine space, we now wish to show that $\phi$ restricts to a set bijection $A_t(s)\to \map{TD}_t(s)$ for $\mcal{C}_s\cap\mcal{C}_t$, as well as the analogous bijection for $\mcal{C}_t$. Suppose $(x_i)\in \{\sum_{i\in\ZZ/t\ZZ} X_i = \binom{t}{2}\}$ corresponds under $\phi$ to $(y_j)\in \{\sum_{j\in\ZZ/t\ZZ} Y_j = s\}$; then we make the following observations.
\begin{enumerate}

\item Since $s,t$ are coprime, the inequalities $x_i \ge x_{i+s} - s$ hold for all $i\in\ZZ/t\ZZ$ if and only if $y_j \ge 0$ for all $j\in\ZZ/t\ZZ$;

\item The identity $\sum_{j=0}^{t-1} jy_{j+\ell} = k - x_{k+s\ell}$ follows (for any $\ell\in\ZZ/t\ZZ$) from the telescoping sum above (substituting $\sum_{j=0}^{t-1} y_{j+\ell} = s$ and $k = \frac{t-1}{2} + \frac{t-1}{2}\cdot s$);

\item If $x_i\equiv i\pmod{t}$ for all $i\in\ZZ/t\ZZ$, then $\sum j y_j = k - x_k \equiv 0\pmod{t}$ and $y_j\equiv0\pmod{1}$ for all $j\in\ZZ/t\ZZ$;

\item Suppose $\sum_{j=0}^{t-1} jy_{j} (= k - x_{k})$ is $0\pmod{t}$, i.e. $x_{k} \equiv k\pmod{t}$, and further $y_{j}\equiv0\pmod{1}$ for all $j\in\ZZ/t\ZZ$. Then for any $\ell\in\ZZ$, we have
\[ x_{k+s\ell} = k - \sum_{j=0}^{t-1} j y_{j+\ell} = k+s\ell - \sum_{j=0}^{t-1} (j+\ell) y_{j+\ell} \equiv k+s\ell - \sum_{j\in\ZZ/t\ZZ} jy_j \equiv k+s\ell\pmod{t}, \]
so $x_i \equiv i\pmod{t}$ for all $i\in\ZZ/t\ZZ$ (as $s,t$ are coprime). (Alternatively, we could look at the partial sums $z_0+\cdots+z_{\ell-1} = \frac1t (x_k - x_{k+\ell s} + \ell s)$.)
\end{enumerate}
The first and third items show that $\phi$ maps $A_t(s)$ into $\map{TD}_t(s)$. The first and fourth items show that $\phi^{-1}$ maps $\map{TD}_t(s)$ into $A_t(s)$. (We are using the fact that $\phi$ bijects a superset of $A_t(s)$ to a superset of $\map{TD}_t(s)$.) So $\phi$ bijects $A_t(s)$ to $\map{TD}_t(s)$, establishing the change of coordinates for $\mcal{C}_s\cap \mcal{C}_t$. Similarly, considering only the third and fourth items (ignoring the first item) establishes the change of coordinates for $\mcal{C}_t$.

Finally, the explicit formula for the inverse $\phi^{-1}$ establishes the desired description of the inverse $(a_{t,i})_{i\in\ZZ/t\ZZ} = \phi^{-1}(z_{t,j})_{j\in\ZZ/t\ZZ}$ of the isomorphism.
\end{proof}

\begin{remark}
Already in this proof we have seen the `cyclic shifts identity' $\sum_{j=0}^{t-1} j y_{j+\ell} \equiv -s\ell + \sum_{j\in\ZZ/t\ZZ} jy_j \pmod{t}$ for $t$ integers $(y_j)_{j\in\ZZ/t\ZZ}$ summing to $s$, which will feature more prominently in Proposition \ref{PROP:cyclic-shift}.
\end{remark}

\begin{corollary}[Weak version of Anderson's theorem \cite{Anderson}]
\label{COR:finitely-many-s,t-cores}

$A_t(s)$ and $\map{TD}_t(s)$ are discrete bounded sets, hence finite. In particular, there are finitely many $(s,t)$-cores, so the sums in Theorems \ref{THM:general-Armstrong}, \ref{THM:self-conjugate-Armstrong}, \ref{THM:general-Fayers}, \ref{THM:self-conjugate-Fayers} are finite and well-defined.
\end{corollary}

\begin{markednewpar}\leavevmode\end{markednewpar}\begin{arxiv}
\begin{sidermk}
\label{RMK:SIDE:Anderson-vs-Johnson-bijections}
As mentioned in the introduction, Anderson \cite{Anderson} actually computes the exact number of $(s,t)$-cores as the `rational Catalan number' $\frac{1}{s+t}\binom{s+t}{s}$, by bijecting the set of $(s,t)$-cores to down-right lattice paths from $(0,s)$ to $(t,0)$ staying below the connecting line (or tuples $(w_1,\ldots,w_{s+t})\in\{s,-t\}^{s+t}$ summing to $0$, with all nonnegative partial sums, where $w_i = s$ corresponds to a rightward step and $w_i = -t$ corresponds to a downward step); these are essentially $(s,t)$-Dyck paths, which Bizley \cite{Bizley} counts via `cyclic shifts of order $s+t$'.

One route to the bijection (mentioned in the introduction) goes from $(s,t)$-cores to beta-sets closed under subtraction by $s,t$, which correspond under negation and suitable translation to subsets of $\ZZ_{\ge0}$ that contain $0$ and are closed under addition by $s,t$, which correspond to the desired lattice paths (see e.g. \cite{ISL}).

At the beginning of our proof of Theorem \ref{THM:general-Armstrong} we implicitly give Johnson's variant using `cyclic shifts of order $t$' in the asymmetric $z$-coordinates. The criterion for $z$-coordinates to correspond $(s,t)$-cores can be geometrically framed as a divisibility condition on ``area under lattice paths from $(0,s)$ to $(t-1,0)$'' modulo $t$, in contrast to the Catalan-like condition for $w$-coordinates.
\end{sidermk}
\end{arxiv}

\subsection{Cyclic shifts in the \texorpdfstring{$z$}{z}-coordinates}
\label{SEC:z-coordinate-cyclic-shifts}

Johnson \cite{Johnson} relies on cyclic symmetry in his proofs of Anderson's theorem \cite{Anderson} (that there are exactly $\frac{1}{s+t}\binom{s+t}{s}$ distinct simultaneous $(s,t)$-cores) and Armstrong's general conjecture (Theorem \ref{THM:general-Armstrong}). For Theorems \ref{THM:general-Armstrong} and \ref{THM:general-Fayers}, we also rely on the following `cyclic shifts' argument.

\begin{proposition}[c.f. {\cite[proofs of Corollary 3.6 and Theorem 3.7]{Johnson}}]
\label{PROP:cyclic-shift}

Fix coprime $s,t\ge1$. Let $f(X_0,\ldots,X_{t-1})$ be a \emph{cyclic} complex-valued function (i.e. for all $i\in\ZZ/t\ZZ$ we have $f(X_0,\ldots,X_{t-1}) = f(X_i,\ldots,X_{i+t-1})$, with indices taken modulo $t$). Then
\[ \sum_{(z_{t,j})_{j\in\ZZ/t\ZZ}\in \map{TD}_t(s)} f(z_{t,0},\ldots,z_{t,t-1}) = \frac{1}{t}\sum_{\substack{x_j\ge0 \\ \sum_{j\in\ZZ/t\ZZ} x_j = s}} f(x_0,\ldots,x_{t-1}). \]
\end{proposition}

\begin{proof}
For any $t$ nonnegative integers $x_0,\ldots,x_{t-1}\ge0$ (indexed modulo $t$) summing to $s$, the cyclic permutations $(x_r,\ldots,x_{r+t-1})$ leave distinct residues via the `cyclic shifts identity'
\[ \sum_{j\in \ZZ/t\ZZ} j x_{r+j} \equiv -rs + \sum_{j\in\ZZ/t\ZZ}(r+j)x_{r+j} \equiv -rs + \sum_{j\in\ZZ/t\ZZ} jx_j \pmod{t}, \]
since $s$ is coprime to $t$. Thus each orbit of the cyclic $\ZZ/t\ZZ$-action contains exactly one point of $\map{TD}_t(s)$, and since $f$ is \emph{cyclic} (and the sums are over the finite sets $\map{TD}_t(s)$ and $\{(x_j)_{j\in\ZZ/t\ZZ}\in\ZZ_{\ge0}^t: \sum_{j\in\ZZ/t\ZZ} x_j = s\}$, by Corollary \ref{COR:finitely-many-s,t-cores}), the result follows.
\end{proof}

\begin{corollary}[Relevant cyclic sums]
\label{COR:special-cyclic-sums}

Fix coprime $s,t\ge1$. Then we evaluate the sum $\sum_{(z_{t,j})\in \map{TD}_t(s)} f(z_{t,0},\ldots,z_{t,t-1})$ in the cases listed below, where the most important terms have been boxed. For a vector or weak composition $\bd{x} = (x_i)_{i\in\ZZ/t\ZZ}$ with $t$ components, define the sum $\abs{\bd{x}} \colonequals \sum_{i\in\ZZ/t\ZZ} x_i$ and (if appropriate) the multinomial coefficient $\binom{\abs{\bd{x}}}{\bd{x}} \colonequals \binom{\abs{\bd{x}}}{x_1,\ldots,x_t}$.

First we look at ``exponential'' cases.
\begin{itemize}
\item \fbox{$\frac1t\cdot t^s$} when $f = \binom{\abs{\bd{x}}}{\bd{x}}\cdot\fbox{$1$}$ (constant);

\item \fbox{$\frac1t\cdot s t^{s-1}$} when $f = \binom{\abs{\bd{x}}}{\bd{x}}\cdot \frac1t\sum_{i\in\ZZ/t\ZZ} \fbox{$x_i$}$ (linear);

\item \fbox{$\frac1t\cdot s(s-1) t^{s-2}$} when $f = \binom{\abs{\bd{x}}}{\bd{x}}\cdot \frac1t\sum_{i\in\ZZ/t\ZZ} \fbox{$x_i(x_i - 1)$}$ or $f = \binom{\abs{\bd{x}}}{\bd{x}}\cdot \frac1t\sum_{i\in\ZZ/t\ZZ} \fbox{$x_i x_{i+r}$}$ for some $r\not\equiv0\pmod{t}$ (quadratic).
\end{itemize}

Next we look at ``ordinary'' cases.

\begin{itemize}
\item \fbox{$\frac1t\cdot\binom{s+t-1}{t-1}$} when $f = \fbox{$1$}$ (constant);

\item \fbox{$\frac1t\cdot\binom{s+t-1}{t}$} when $f = \frac1t\sum_{i\in\ZZ/t\ZZ} \fbox{$x_i$}$ (linear);

\item \fbox{$\frac1t\cdot2\binom{s+t-1}{t+1}$} when $f = \frac1t\sum_{i\in\ZZ/t\ZZ} \fbox{$x_i(x_i-1)$}$ (square quadratic);

\item \fbox{$\frac1t\cdot\binom{s+t-1}{t+1}$} when $f = \frac1t\sum_{i\in\ZZ/t\ZZ} \fbox{$x_i x_{i+r}$}$ for some $r\not\equiv0\pmod{t}$ (mixed quadratic).
\end{itemize}
\end{corollary}

\begin{remark}
We will use these explicit (generating function) calculations below in the proofs of Theorems \ref{THM:general-Armstrong} and \ref{THM:general-Fayers}, instead of the coarser Ehrhart and Euler--Maclaurin theory language of Johnson \cite{Johnson}. (See Remarks \ref{RMK:conceptual-proof} and \ref{RMK:higher-power-methods} for further conceptual discussion.) \begin{arxiv}(However, the difference is probably not as large as it might sound: Ehrhart theory simply captures many of the most important qualitative aspects of the generating functions related to lattice point geometry; see the textbook \cite{Ehrhart} for more details.)\end{arxiv}
\end{remark}

\begin{proof}[Proof of ``exponential'' cases]
First we use Proposition \ref{PROP:cyclic-shift} to reduce the cyclic sums in question over $\map{TD}_t(s)$ to sums over the easier domain of $\{x_i\ge0: \abs{\bd{x}} = s\}$. On this easier domain, the standard tool of exponential generating functions ($\prod_{i\in\ZZ/t\ZZ} \exp(Z_i T)$) suffices. Equivalently, one may directly differentiate the multinomial $\sum_{\abs{\bd{x}} = s} \binom{\abs{\bd{x}}}{\bd{x}} Z_0^{x_0} \cdots Z_{t-1}^{x_{t-1}} = (Z_0+\cdots+Z_{t-1})^s$. For example, by differentiating zero times we get $\sum_{\abs{\bd{x}} = s} \binom{\abs{\bd{x}}}{\bd{x}} = t^s$; by differentiating once we get $\sum_{\abs{\bd{x}} = s} \binom{\abs{\bd{x}}}{\bd{x}} x_0 = s t^{s-1}$; and by differentiating twice we get $\sum_{\abs{\bd{x}} = s} \binom{\abs{\bd{x}}}{\bd{x}} x_0(x_0 - 1) = \sum_{\abs{\bd{x}} = s} \binom{\abs{\bd{x}}}{\bd{x}} x_0 x_r = s(s-1) t^{s-2}$ (if $t\nmid r$).
\end{proof}

\begin{proof}[Proof of ``ordinary'' cases]
Once again we use Proposition \ref{PROP:cyclic-shift} to reduce the cyclic sums in question over $\map{TD}_t(s)$ to sums over the easier domain of $\{x_i\ge0: \abs{\bd{x}} = s\}$, where the standard tool of ordinary generating functions ($\prod_{i\in\ZZ/t\ZZ} (1-Z_i T)^{-1}$) suffices. For example, by differentiating zero times we get $\sum_{\abs{\bd{x}} = s} 1 = [T^s](1-T)^{-t} = \binom{s+t-1}{t-1}$; by differentiating once (with respect to $Z_0$) we get $\sum_{\abs{\bd{x}} = s} x_0 = [T^s]T(1-T)^{-t-1} = \binom{(s-1)+(t+1)-1}{(t+1)-1}$; by differentiating twice with respect to $Z_0$ we get $\sum_{\abs{\bd{x}} = s} x_0(x_0 - 1) = [T^s]2T^2(1-T)^{-t-2} = 2\binom{s+t-1}{t+1}$; or by differentiating once with respect to each of $Z_0,Z_r$ (if $t\nmid r$) we get $\sum_{\abs{\bd{x}} = s} x_0 x_r = [T^s]T^2(1-T)^{-t-2} = \binom{s+t-1}{t+1}$.
\end{proof}

\subsection{Size of a \texorpdfstring{$t$}{t}-core}
\label{SEC:z-coordinates-partition-size}

The $c$- and $a$- coordinate versions of the following lemma could have been placed earlier, but we wish to use the $z$-coordinates, most relevant right before Section \ref{SEC:main-theorem-proofs}. The $c$-coordinate version has been used in \cite{Cranks-t-cores-expos} and \cite{Physics-expos} to deduce certain properties of the generating function for $t$-cores.

\begin{lemma}[c.f. {\cite[proof of Theorem 2.10]{Johnson}}]
\label{LEM:size-of-partition-z-coordinates}

Fix $t\ge1$ and a $t$-core $\lambda$.
\begin{itemize}
\item In Johnson's $c_t$-coordinates, $\card{\lambda} = \sum_{i=0}^{t-1}(\frac{t}{2} c_{t,i}^2 - (\frac{t-1}{2} - i)c_{t,i})$, from \cite[Theorem 2.10]{Johnson}; \cite[Bijection 2]{Cranks-t-cores-expos}; \cite{Physics-expos}, up to the linear relation $\sum_{i\in\ZZ/t\ZZ} c_{t,i} = 0$.

\item In Fayers' $a_t$-coordinates, $\card{\lambda} = -\frac{1}{24}(t^2-1) +  \frac{1}{2t}\sum_{i\in\ZZ/t\ZZ} [a_{t,i}-\frac{t-1}{2}]^2$ (a \emph{symmetric} polynomial);

\item Fix $s\ge1$ coprime to $t$. In the extended $z_t$-coordinates,
\mathd{\card{\lambda} = -\frac{1}{24}(t^2-1) + \frac{1}{24}(t^2-1)\sum_{\ell\in\ZZ/t\ZZ}z_{t,\ell}^2 + M_2(z_{t,0},\ldots,z_{t,t-1})}
(a \emph{cyclic} polynomial), where $M_2\in\ZZ[X_0,\ldots,X_{t-1}]$ is the ``leftover'' \emph{cyclic} homogeneous quadratic with only `mixed' terms (i.e. no square terms $X_0^2,\ldots,X_{t-1}^2$), and with coefficient sum $-\frac{1}{24}t(t^2-1)$.
\end{itemize}
In particular, in $z$-coordinates, the sum of the coefficients of $\card{\lambda} + \frac{1}{24}(t^2-1)$---the non-constant (i.e. homogeneous quadratic) part of the $z$-expression---is $0$.
\end{lemma}

\begin{markednewpar}\leavevmode\end{markednewpar}\begin{arxiv}
\begin{sidermk}
We will only explicitly use this result for $(s,t)$-cores $\lambda\in \mcal{C}_s\cap \mcal{C}_t$, when $s,t$ are coprime, in the proof of Theorems \ref{THM:general-Armstrong} and \ref{THM:general-Fayers}. But the general $t$-core version may help for other problems.
\end{sidermk}

\begin{sidermk}
In principle we could compute the coefficients of $M_2$ more explicitly, but by the computations in Corollary \ref{COR:special-cyclic-sums}, we will not need to explicitly distinguish the different terms of $M_2$.
\end{sidermk}
\end{arxiv}

\begin{remark}
Our proofs of Theorems \ref{THM:general-Armstrong} and \ref{THM:general-Fayers} rely crucially on the cyclic symmetry of the formula for the size $\card{\lambda}$ in the $z$-coordinates. (Johnson's proof of Theorem \ref{THM:general-Armstrong} does as well \cite{Johnson}.) Due to the cyclic nature of the change of variables relating the $z$- and $a$- coordinates (see Proposition \ref{PROP:change-to-z-coordinates}), it suffices to understand why $\card{\lambda}$ is symmetric in the $a$-coordinates.
\end{remark}

The following proof directly explains the (complete) symmetry in the $a$-coordinates, but one could also convert from the $c$-coordinates using Section \ref{SEC:s-sets-and-a-coordinates-vs-c-coordinates}.

\begin{proof}[Discrete calculus proof for $a$-coordinates]

Split the partition into two halves by cutting along the $y$-axis (see Figure \ref{FIG:johnson-russian-notation-area-computation}). We calculate the size of the partition (or area of the partition diagram) as
\[ \card{\lambda} = \sum_{\substack{x\in B^\lambda \\ x+\frac12 > 0}} \left(x + \frac12\right) - \sum_{\substack{x\notin B^\lambda \\ x+\frac12 < 0}} \left(x + \frac12\right). \]

\begin{figure}
\centering

\trimbox{0cm 1.7cm 0cm 0cm}{
\begin{tikzpicture}
\begin{scope}[gray, ultra thin, scale=.6] 
\clip (-5.5, 5.5) rectangle (5.5, -5);
\draw[rotate=45, scale=1.412] (0,0) grid (6,6);
\end{scope}

\begin{scope}[rotate=45, scale=.6*1.412] 
\draw (0,5.5) -- (0, 3) -- (1,3) -- (1,2) -- (3,2) -- (3,0) -- (5.5,0);
\end{scope}

\begin{scope}[rotate=45, ultra thick, scale=.6*1.412]
\draw (0, 3) -- (1,3);
\draw (1,2) -- (2,2);
\draw (3,2) -- (3,0);
\end{scope}



\begin{scope}[scale=.6, yshift=-.5cm]
\draw[thick, dotted] (0,-.5) -- (0,6.5); 
\draw (-5.5,0) node{$\cdots$};
\draw (-4.5,0) circle (.3) node[below=3pt]{$\frac{9}{2}$};
\draw (-3.5,0) circle (.3) node[below=3pt]{$\frac{7}{2}$};
\filldraw (-2.5,0) circle (.3) node[below=3pt]{$\frac{5}{2}$};
\draw (-1.5,0) circle (.3) node[below=3pt]{$\frac{3}{2}$};
\filldraw (-.5,0) circle (.3) node[below=3pt]{$\frac{1}{2}$};
\filldraw (.5,0) circle (.3) node[below=3pt]{$\frac{-1}{2}$};
\draw (1.5,0) circle (.3) node[below=3pt]{$\frac{-3}{2}$};
\draw (2.5,0) circle (.3) node[below=3pt]{$\frac{-5}{2}$};
\filldraw (3.5,0) circle (.3) node[below=3pt]{$\frac{-7}{2}$};
\filldraw (4.5,0) circle (.3) node[below=3pt]{$\frac{-9}{2}$};
\draw (5.5,0) node{$\cdots$};
\end{scope}
\end{tikzpicture}
}

\caption{\emph{Russian notation} for $\lambda = (3,2,2,0,\ldots)$, based on Figure \ref{FIG:johnson-russian-notation}. The $y$-axis (dotted) is distinguished for area computation: $\card{\lambda} = \frac52 + \frac12 -(\frac{-3}{2}) - (\frac{-5}{2})$.} \label{FIG:johnson-russian-notation-area-computation}
\end{figure}

As written this holds for any partition, but we want to simplify it further for $t$-cores $\lambda$. As we are working with the $a_{t,i}$-coordinates, we will break up the contributions by residue class modulo $t$. Fix $0\le i\le t-1$, and define the $i+t\ZZ$ area contribution $F(a_{t,i})$ (i.e. contribution in the sums above from $x$ congruent to $i\pmod{t}$); we will study how it changes as $a_{t,i}$ varies along the residue class $i+t\ZZ$. We have
\begin{itemize}
\item $F(i) = 0$ (no contribution when $B^\lambda\cap(i+t\ZZ) = \{\ldots,i-3s,i-2s,i-s\}$);

\item $F(a_{t,i}) - F(a_{t,i} - t) = [a_{t,i} - t] + \frac12$ (note that this holds both in the `positive/left case' $a_{t,i} \ge i+t$ and the `negative/right case' $a_{t,i} \le i$).
\end{itemize}
Standard discrete calculus yields $F(a_{t,i}) = G(a_{t,i}) - G(i)$, where $G(x) = \frac1{2t} x(x-t) + \frac12\cdot \frac1t x$. Completing the square gives $F(a_{t,i}) = \frac{1}{2t}[a_{t,i} - \frac{t-1}{2}]^2 - \frac{1}{2t}[i - \frac{t-1}{2}]^2$.

To finish, we sum over all $0\le i\le t-1$, using the identity $\frac{1}{2t}\sum_{i=0}^{t-1} (i - \frac{t-1}{2})^2 = \frac{1}{24}(t^2-1)$.
\end{proof}

\begin{markednewpar}\leavevmode\end{markednewpar}\begin{arxiv}
\begin{proof}[Discrete calculus proof for $c$-coordinates]
Mimicking the proof in the $a$-coordinates, the area contribution from $c_{t,i}$ is $(i+\frac12) + \sum_{u=0}^{-c_{t,i}} (tu - i - \frac12) = t\binom{-c_{t,i}+1}{2} - (i+\frac12) \binom{-c_{t,i}}{1} = \frac{t}{2} c_{t,i}^2 - (\frac{t-1}{2} - i)c_{t,i}$, where the $i+\frac12$ in the front is just a correction term, as when $-c_{t,i} = 0$, the contribution should be $0$. To finish, we sum the contributions over all $0\le i\le t-1$.
\end{proof}
\end{arxiv}

\begin{proof}[Proof of $a$-to-$z$ translation]

Write $\card{\lambda} = -\frac{1}{24}(t^2-1) + \frac{1}{2t}\cdot P$ for convenience. Proposition \ref{PROP:change-to-z-coordinates} (which holds for $t$-cores, not just $(s,t)$-cores) gives $a_{t,k+\ell s} - \frac{t-1}{2} = \sum_{j=0}^{t-1} (\frac{t-1}{2} - j) z_{t,j+\ell}$ for all $\ell\in\ZZ/t\ZZ$, since $s,t$ are coprime. Then
\begin{equation}\label{EQ:size-a-to-z-translation}
P\colonequals \sum_{i=0}^{t-1} \left[ a_{t,i} - \frac{t-1}{2} \right]^2
= \sum_{\ell\in\ZZ/t\ZZ} \left[ a_{t,k+\ell s} - \frac{t-1}{2} \right]^2
= \sum_{\ell\in\ZZ/t\ZZ} \left[ \sum_{j=0}^{t-1} \left(\frac{t-1}{2} - j \right) z_{t,j+\ell} \right]^2
\end{equation}
is a cyclic quadratic polynomial in the $z_{t,\ell}$, so $P = C(t)\sum_{\ell\in\ZZ/t\ZZ}z_{t,\ell}^2 + H(z_{t,0},\ldots,z_{t,t-1})$ for some constant $C$ depending only on $t$, and a cyclic homogeneous quadratic $H\in \ZZ[X_0,\ldots,X_{t-1}]$ with no `square terms' (i.e. $X_0^2,\ldots,X_{t-1}^2$). We make the following calculations.
\begin{itemize}
\item $P$ vanishes when $a_{t,\ell} = \frac{t-1}{2}$ for all $\ell\in\ZZ/t\ZZ$, or equivalently when the $z$-coordinates are all equal. So the sum of $z$-coefficients of $\card{\lambda} + \frac{1}{24}(t^2-1) = \frac{1}{2t}\cdot P$ is $0$.

\item $C(t) = \sum_{j=0}^{t-1} (\frac{t-1}{2} - j)^2 = \frac{1}{12} t(t^2-1)$.

\item $H(1,\ldots,1) = P(1,\ldots,1) - C(t)\sum_{\ell\in\ZZ/t\ZZ}1^2 = 0 - t C(t) = -\frac{1}{12} t^2(t^2-1)$.
\end{itemize}
Substituting $P$ into $\card{\lambda} = -\frac{1}{24}(t^2-1) + \frac{1}{2t}\cdot P$ finishes the job.
\end{proof}

\section{Proofs of general conjectures}
\label{SEC:main-theorem-proofs}

In this section, we first give a proof of Armstrong's general conjecture by direct computation. We then use the same methods to prove Fayers' general conjecture.

\begin{proof}[Proof of Theorem \ref{THM:general-Armstrong} by direct computation]
We use $z$-coordinates. By Corollary \ref{COR:special-cyclic-sums}, the denominator is just $\sum_{\map{TD}_t(s)}1 = \frac1t\binom{s+t-1}{t-1}$ (the number of $(s,t)$-cores).

Using Lemma \ref{LEM:size-of-partition-z-coordinates} and Corollary \ref{COR:special-cyclic-sums} with the identity $x^2 = x(x-1) + x$, the numerator $\sum_{\map{TD}_t(s)}\card{\lambda} = \sum_{\map{TD}_t(s)}(-\frac{1}{24}(t^2-1)\cdot 1 + \frac{1}{24}(t^2-1)\cdot\sum_{\ell\in\ZZ/t\ZZ} z_{t,\ell}^2 + M_2(z_{t,\ell}))$ becomes
\begin{align*}
&-\frac{t^2-1}{24}\cdot\frac1t\binom{N}{t-1} + \frac{t^2-1}{24}\cdot t\cdot \frac1t \left[\binom{N}{t} + 2\binom{N}{t+1}\right] - \frac{t(t^2-1)}{24}\cdot\frac1t\binom{N}{t+1} \\
&= \frac{t^2-1}{24}\cdot\frac1t\binom{N}{t-1} \left[-1 + t\cdot\frac{s}{t} + t\cdot\frac{s(s-1)}{t(t+1)} \right] ,
\end{align*}
where we have suppressed $N\colonequals s+t-1$ and used $\binom{N}{t} = \frac{s}{t}\binom{N}{t-1}$ and $\binom{N}{t+1} = \frac{s-1}{t+1}\binom{N}{t}$ (viewed as polynomial identities in $s$, for fixed $t\ge1$). Finally, dividing the numerator expression by the denominator $\frac1t\binom{N}{t-1}$ yields $\frac{1}{24}(t^2-1)(-1+s+\frac{s(s-1)}{t+1})$, which simplifies to $\frac{1}{24}(t-1)(t+1)\frac{s-1}{t+1}[(t+1)+s] = \frac{1}{24}(s-1)(t-1)(s+t+1)$.
\end{proof}

The same technique proves Fayers' general conjecture.

\begin{proof}[Proof of Theorem \ref{THM:general-Fayers}]
In $z$-coordinates, the stabilizer $\stab_{G_{s,t}}(\lambda)$ has size $\prod_{i\in\ZZ/t\ZZ} z_{t,i}!$ (by Corollary \ref{COR:s-set-and-stabilizer-data-in-z-coordinates}). By Corollary \ref{COR:special-cyclic-sums}, the denominator times $s!$ is just $\sum_{\map{TD}_t(s)}\binom{\abs{\bd{z_t}}}{\bd{z_t}} = \frac1t\cdot t^s = t^{s-1}$.

Using Lemma \ref{LEM:size-of-partition-z-coordinates} and Corollary \ref{COR:special-cyclic-sums} with the identity $x^2 = x(x-1) + x$, the numerator times $s!$, i.e. $\sum_{\map{TD}_t(s)}\binom{\abs{\bd{z_t}}}{\bd{z_t}}\card{\lambda} = \sum_{\map{TD}_t(s)}\binom{\abs{\bd{z_t}}}{\bd{z_t}}(-\frac{1}{24}(t^2-1)\cdot 1 + \frac{1}{24}(t^2-1)\cdot\sum_{\ell\in\ZZ/t\ZZ} z_{t,\ell}^2 + M_2(z_{t,\ell}))$, becomes
\[
-\tfrac{1}{24}(t^2-1)\cdot\tfrac1t\cdot t^s + \tfrac{1}{24}(t^2-1)\cdot [ st^{s-1} + s(s-1)t^{s-2} ] - \tfrac{1}{24}(t^2-1)\cdot s(s-1)t^{s-2},
\]
which simplifies to $\frac{1}{24}(t^2-1) t^{s-1} (s-1)$. Finally, dividing ($s!$ times the) numerator by ($s!$ times the) denominator yields the desired ratio of $\frac{1}{24}(s-1)(t^2-1)$.
\end{proof}

\begin{remark}
\label{RMK:conceptual-cancellation-Fayers-general}

Conceptually, the $s(s-1)t^{s-2}$ coefficients cancel because they collect to form the sum of the coefficients of $\card{\lambda} + \frac{1}{24}(t^2-1)$, which is $0$ as noted in Lemma \ref{LEM:size-of-partition-z-coordinates}.
\end{remark}

\section{Self-conjugate analogs}
\label{SEC:self-conjugate-analogs}

Using almost the same methods as before, we build up to proofs of Theorems \ref{THM:self-conjugate-Armstrong} and \ref{THM:self-conjugate-Fayers} in Section \ref{SEC:main-theorem-proofs-self-conjugate}. However, we no longer need a cyclic shifts argument, because the parameterization of self-conjugate cores (see Proposition \ref{PROP:change-to-z-coordinates-self-conjugate}) is simpler.

\subsection{Background: beta-sets, charges, \texorpdfstring{$s$}{s}-sets, and conjugation}
\label{SEC:beta-set-and-s-charges-of-conjugate-partition}

\begin{proposition}[c.f. {\cite[Lemma 4.3]{Fayers}}]
\label{PROP:conjugate-beta-set}

Fix $x\in\ZZ$.\begin{arxiv} See Figure \ref{FIG:johnson-russian-notation}. \end{arxiv} Then $x\in B^{\conj{\lambda}}$\begin{arxiv} if and only if $x+\frac12$ slopes upwards in $\conj{\lambda}$; if and only if $-x-\frac12$ slopes downwards in $\lambda$;\end{arxiv} if and only if $-1-x\notin B^\lambda$.
\end{proposition}

\begin{proposition}[Extension of Fayers {\cite[Lemma 4.6]{Fayers}}]
\label{PROP:conjugate-charge}

Fix $s\ge1$ and $i\in\ZZ/s\ZZ$. Then $c_{s,i}(\conj{\lambda}) = -c_{s,-1-i}(\lambda)$ holds for all partitions $\lambda$.
\end{proposition}

\begin{proof}
Substituting Proposition \ref{PROP:conjugate-beta-set} into Definition \ref{REF:c-coordinates-charge-definition} for $c_{s,i}(\conj{\lambda})$ gives
\[
c_{s,i}(\conj{\lambda})
=
\#[(-1-i+s\ZZ)\cap(-1-B^\lambda)\cap\{x+\tfrac12<0\}]
-
\#[(-1-i+s\ZZ)\cap (\ZZ\setminus -1-B^\lambda)\cap\{x+\tfrac12>0\}].
\]
Applying the involution $x\mapsto -1-x$ recovers the definition of $-c_{s,-1-i}(\lambda)$.
\end{proof}

But the $s$-core operation preserves $s$-charge (see Proposition \ref{REF:s-core-operation-is-well-defined}), so Proposition \ref{PROP:conjugate-charge} yields $c_{s,i}(\conj{\lambda}^s) = c_{s,i}(\conj{\lambda}) = c_{s,-1-i}(\lambda) = c_{s,-1-i}(\lambda^s) = c_{s,i}(\conj{\lambda^s})$. Since $c$-coordinates uniquely parameterize $s$-cores (Proposition \ref{REF:c-coordinates-parameterization-of-s-cores}), we obtain the following result.

\begin{corollary}
\label{COR:conjugation-commutes-with-s-core-operation}

The $s$-core operation commutes with conjugation, i.e. $\conj{\lambda}^s = \conj{\lambda^s}$ for any partition $\lambda$.
\end{corollary}

The $c$-to-$a$ translation of Section \ref{SEC:s-sets-and-a-coordinates-vs-c-coordinates} gives another corollary of Proposition \ref{PROP:conjugate-charge}.

\begin{corollary}[{\cite[proof of Lemma 4.6]{Fayers}}]
\label{COR:conjugate-s-set}

Fix $\lambda\in \mcal{C}_s$. Then $\conj{\lambda}\in \mcal{C}_s$ as well, and $a_{s,i}(\conj{\lambda}) = s-1-a_{s,-1-i}(\lambda)$ for all $i\in\ZZ/s\ZZ$, or equivalently $\mcal{S}_s(\conj{\lambda}) = s-1 - \mcal{S}_s(\lambda)$. In particular, $\lambda\in \mcal{D}_s$ if and only if $\mcal{S}_s(\lambda) = s-1-\mcal{S}_s(\lambda)$, i.e. $\mcal{S}_s$ is a \emph{symmetric} $s$-set in Fayers' terminology.
\end{corollary}

\begin{markednewpar}\leavevmode\end{markednewpar}\begin{arxiv}
\begin{sidermk}
Fayers gives a more direct proof of this fact in {\cite[proof of Lemma 4.6]{Fayers}}. If $\lambda$ is an $s$-core, then (suppressing notation for clarity)
\[ B^\lambda = \bigcup_{i\in\ZZ/s\ZZ}\{a_i - s, a_i - 2s, a_i - 3s,\ldots\} = \ZZ\setminus \bigcup_{i\in\ZZ/s\ZZ}\{a_i, a_i + s, a_i + 2s,\ldots\}, \]
so for the conjugate $\conj{\lambda}$ we have (again suppressing notation for clarity)
\begin{align*}
B^{\conj{\lambda}}
&= \bigcup_{i\in\ZZ/s\ZZ}\{-1-a_i, -1-a_i-s, -1-a_i-2s, \ldots\} \\
&= \bigcup_{i\in\ZZ/s\ZZ}\{-1-a_{-1-i}, -1-a_{-1-i}-s, -1-a_{-1-i}-2s, \ldots\}.
\end{align*}
Thus $\conj{\lambda}$ is also an $s$-core (which also follows from the original hook length definition), with $a_{s,i}(\conj{\lambda}) = s-1-a_{s,-1-i}(\lambda)$ for all $i\in\ZZ/s\ZZ$.
\end{sidermk}
\end{arxiv}

\subsection{Background: \texorpdfstring{$t$}{t}-cores of self-conjugate \texorpdfstring{$s$}{s}-cores}
\label{SEC:t-core-of-self-conjugate-s-core-exposition}

\subsubsection{Comparing \texorpdfstring{$t$}{t}-cores of different self-conjugate \texorpdfstring{$s$}{s}-cores}

The only necessary modification from the general case is the following self-conjugate version of Proposition \ref{PROP:extending-Olsson-theorem}.

\begin{proposition}[Extension of Proposition \ref{PROP:extending-Olsson-theorem}]
\label{PROP:extending-Olsson-theorem-self-conjugate}

Fix \emph{any} $s,t\ge1$ and $\lambda\in \mcal{D}_s$. Then $\lambda^t\in \mcal{D}_s\cap \mcal{D}_t$, and furthermore $\mcal{S}_s(\lambda^t) \equiv \mcal{S}_s(\lambda) \pmod{t}$ (viewed as multisets of residues).
\end{proposition}

\begin{proof}[Direct alternative proof]
Proposition \ref{PROP:extending-Olsson-theorem} already shows everything except that $\lambda^t$ is self-conjugate. Strictly speaking, \cite[Proposition 4.7]{Fayers} does prove this indirectly, but it is cleaner to use $\lambda = \conj{\lambda}$ and Corollary \ref{COR:conjugation-commutes-with-s-core-operation} to get $\lambda^t = \conj{\lambda}^t = \conj{\lambda^t}$.
\end{proof}

\subsubsection{Self-conjugate level \texorpdfstring{$t$}{t} action}

We now present a self-conjugate analog of the group $G_{s,t}$ from Definition \ref{DEF:level-t-group-action}.

\begin{definition}[c.f. {\cite[Section 4.1]{Fayers}}]
\label{DEF:level-t-group-action-self-conjugate}

Fix coprime $s,t\ge1$. Let $H_{s,t}$ be the subgroup of $G_{s,t}$ consisting of $f\in G_{s,t}$ such that $f(-1-m) = -1-f(m)$ for all $m\in\ZZ$ (based on the involution $m\mapsto -1-m$).

It is easy to check that $H_{s,t}$ has a group structure, and that it acts in the obvious way on the set of \emph{symmetric} $s$-sets $\mcal{S}_s(\lambda)$, or equivalently suitable beta-sets. This induces an action on $\mcal{D}_s$, via Corollary \ref{COR:conjugate-s-set}.
\end{definition}

\begin{remark}[Different but equivalent definitions]
\label{RMK:difference-of-group-definitions-self-conjugate}

We have not given Fayers' actual definition of the `level $t$ action of the $s$-affine hyperoctahedral group', but rather one equivalent by \cite[Propositions 4.1 and 4.2]{Fayers}, and easier to work with for our purposes.
\end{remark}

\begin{remark}[Natural action?]
It is not hard to show that $H_{s,t}$ is the subset of elements $f\in G_{s,t}$ that preserve $\mcal{D}_s$.\begin{arxiv} Indeed, suppose $f\mcal{D}_s \subseteq \mcal{D}_s$. We want to prove $f(s-1-m) = s-1-f(m)$ for every $m\in\{0,1,\ldots,s-1\}$.

\begin{itemize}
\item If $m \ne s-1-m$, then for all positive integers $N$, consider the action of $f$ on the symmetric $s$-set $\{m+Ns,s-1-m-Ns\}\cup \{0,1,\ldots,s-1\}\setminus\{m,s-1-m\}$. If $N$ is sufficiently large, then the only way the image can be symmetric is for $f(m)+Ns$ to pair up with $f(s-1-m)-Ns$ to sum to $s-1$, so that $f(m)+f(s-1-m) = s-1$.

\item If $m = s-1-m$, i.e. $s$ is odd and $m = \frac{s-1}{2}$, then by the previous point, we must have $f(n) = s-1-f(s-1-n)$ for all $n\in\{0,1,\ldots,s-1\}\setminus\{\frac{s-1}{2}\}$. But then $f(m) = \binom{s}{2} - \sum_{n\ne m} f(n) = \frac{s-1}{2}$, so that $f(m)+f(s-1-m) = 2f(m) = s-1$.
\end{itemize}
\end{arxiv}
\end{remark}

We have following self-conjugate analog of Proposition \ref{PROP:extending-bijections-of-s-sets-to-group-actions}, with analogous proof.

\begin{proposition}[c.f. {\cite[proof of Proposition 4.9]{Fayers}}]
\label{PROP:extending-bijections-of-s-sets-to-group-actions-self-conjugate}

Let $\lambda,\mu$ be self-conjugate $s$-cores, and $\phi\colon \mcal{S}_s(\lambda)\to \mcal{S}_s(\mu)$ a bijection (of \emph{symmetric} $s$-sets) that preserves residue classes modulo $t$, satisfying $\phi(m)+\phi(s-1-m) = s-1$ for all $m\in \mcal{S}_s(\lambda)$. Then $\phi$ uniquely extends to an element $f\in H_{s,t}$ (under Definition \ref{DEF:level-t-group-action-self-conjugate}).
\end{proposition}

\begin{markednewpar}\leavevmode\end{markednewpar}\begin{arxiv}
\begin{proof}
By Proposition \ref{PROP:extending-bijections-of-s-sets-to-group-actions}, the permutation $\phi$ extends to a unique element $f\colon\ZZ\to\ZZ$ of $G_{s,t}$. We need to show $f(m)+f(-1-m) = -1$ for all integers $m$, or equivalently (by $s$-periodicity in $G_{s,t}$) $f(m)+f(s-1-m) = s-1$ for all integers $m$. This follows (by $s$-periodicity) from our assumption that $\phi(m)+\phi(s-1-m) = s-1$ for all $m$ in the symmetric $s$-set $\mcal{S}_s(\lambda)$.
\end{proof}
\end{arxiv}

We can now state a self-conjugate analog of Corollary \ref{COR:level-t-group-action-orbits}.

\begin{corollary}[{\cite[Corollary 4.8]{Fayers}}]
\label{COR:level-t-group-action-orbits-self-conjugate}

Fix coprime $s,t\ge1$. Then $\lambda,\mu\in \mcal{D}_s$ lie in the same $H_{s,t}$-orbit if and only if $\lambda^t = \mu^t$. In other words, each $H_{s,t}$-orbit of $\mcal{D}_s$ contains a unique $t$-core (hence self-conjugate $(s,t)$-core), and any $\lambda\in \mcal{D}_s$ lies in $H_{s,t}\lambda^t$.
\end{corollary}

\begin{proof}[Sketch of more direct proof]

Fix $\lambda,\mu\in \mcal{D}_s$. Proposition \ref{PROP:faithful-invariant} followed by Proposition \ref{PROP:extending-bijections-of-s-sets-to-group-actions-self-conjugate} gives equivalence of $\lambda^t = \mu^t$ and $\mu \in H_{s,t}\lambda$. The re-phrasing follows by specializing to $\mu \colonequals \lambda^t$ (where $\lambda^t \in \mcal{D}_s$ follows from Proposition \ref{PROP:extending-Olsson-theorem-self-conjugate}).
\end{proof}

\begin{markednewpar}\leavevmode\end{markednewpar}\begin{arxiv}
\begin{proof}[Full direct proof]

Let $\lambda,\mu$ be two self-conjugate $s$-cores. Proposition \ref{PROP:faithful-invariant} states that $\lambda^t = \mu^t$ if and only if we have a congruence of (symmetric) $s$-sets $\mcal{S}_s(\lambda) \equiv \mcal{S}_s(\mu)\pmod{t}$, i.e. there exists a bijection $\phi\colon \mcal{S}_s(\lambda) \to \mcal{S}_s(\mu)$ (of \emph{symmetric} $s$-sets) preserving residues modulo $t$. Since $\mcal{S}_s(\lambda)$ and $\mcal{S}_s(\mu)$ are \emph{symmetric} $s$-sets, we may further specify this bijection $\phi$ to satisfy $\phi(m)+\phi(s-1-m) = s-1$ for all $m\in \mcal{S}_s(\lambda)$. By Proposition \ref{PROP:extending-bijections-of-s-sets-to-group-actions-self-conjugate}, such bijections $\phi$ correspond to group elements $f\in H_{s,t}$ with restriction $f|_{\mcal{S}_s(\lambda)} = \phi$. So $\lambda^t = \mu^t$ if and only if $\mu = f\lambda$ for some $f\in H_{s,t}$, i.e. $\lambda,\mu$ lie in the same $H_{s,t}$-orbit.

Finally, for any self-conjugate $s$-core $\lambda\in \mcal{D}_s$, we have $\lambda^t\in G_{s,t}\lambda$ by specializing the previous paragraph to $\mu\colonequals \lambda^t\in \mcal{D}_s$.
\end{proof}
\end{arxiv}

Definition \ref{DEF:level-t-group-action-self-conjugate} and Corollary \ref{COR:level-t-group-action-orbits-self-conjugate} provide the background and context for Theorem \ref{THM:self-conjugate-Fayers} (stated in the introduction).

\subsection{Key inputs for computation}
\label{SEC:computation-inputs-self-conjugate}

In this section, we describe all the computational methods and results used to compute the sums in Theorems \ref{THM:self-conjugate-Armstrong} and \ref{THM:self-conjugate-Fayers}, roughly following the structure of Section \ref{SEC:computation-inputs}.

\subsubsection{Johnson's \texorpdfstring{$u$}{u}-coordinates versus Fayers' \texorpdfstring{$t$}{t}-sets}
\label{SEC:z-coordinates-intro-self-conjugate}

We will use the $z$-coordinates again (and the closely related $u$-coordinates) to get a simple expression (Proposition \ref{PROP:stab-computation-self-conjugate}) for the $H_{s,t}$-stabilizers. Again, we not only review Johnson's $(s,t)$-core parameterization, but also give a modest extension to general $t$-cores.

\begin{proposition}[c.f. {\cite[Section 3.3]{Johnson}}: Johnson's $u$-coordinates versus Fayers' $t$-sets]
\label{PROP:change-to-z-coordinates-self-conjugate}

Fix coprime $s,t\ge1$. For convenience, let $s' = \floor{s/2}$ and $t' = \floor{t/2}$. The set $\mcal{D}_s\cap\mcal{D}_t$ of self-conjugate $(s,t)$-cores (viewed as self-conjugate $s$-cores within the set of $t$-cores) is parameterized by either of the sets $B_t(s),U_t(s)$, described as follows.
\begin{itemize}
\item By Lemma \ref{LEM:c,a-coordinate-parameterization-s,t-cores}, the set of $t$-sets $\mcal{S}_t(\lambda)$ of self-conjugate $(s,t)$-cores $\lambda$, i.e. the subset $B_t(s)$ of points $(a_{t,i})_{i\in\ZZ/t\ZZ}\in A_t(s)$ (specializing $a$-coordinates from Proposition \ref{PROP:change-to-z-coordinates}) satisfying the additional symmetry $a_{t,i} + a_{t,-1-i} = t-1$ from Corollary \ref{COR:conjugate-s-set}---this is \cite[Lemma 4.6]{Fayers}.

\item The subset $U_t(s)$ of points $(z_{t,i})_{i\in\ZZ/t\ZZ} \in \map{TD}_t(s)$ (specializing $z$-coordinates from Proposition \ref{PROP:change-to-z-coordinates}) satisfying the additional symmetry $z_{t,i} = z_{t,-i}$---this is \cite[Lemma 3.9]{Johnson}.

\item In fact, as implicitly used in \cite[proof of Lemma 3.10]{Johnson}, the set $U_t(s)$ is canonically isomorphic to the set of lattice points $\{(u_0,u_1,\ldots,u_{t'})\in \ZZ_{\ge0}^{1+t'}: u_0+u_1+\cdots+u_{t'} = s'\}$, with the isomorphism defined by $u_0 \colonequals \floor{z_{t,0}/2} = s' - \sum_{i=1}^{t'} u_i$ and the following additional relations: $u_i \colonequals z_{t,i}$ (for $1\le i\le t'$) if $t$ is odd; and $u_{t'} \colonequals z_{t,t'}/2$ and $u_i \colonequals z_{t,i}$ (for $1\le i\le t' - 1$) if $t$ is even.
\end{itemize}

The isomorphism between these sets (also preserving the ambient linear and simplex structures) is the same as that described in Proposition \ref{PROP:change-to-z-coordinates}.

Under the same affine change of coordinates, the set $\mcal{D}_t$ of self-conjugate $t$-cores can be parameterized as follows.
\begin{itemize}
\item By Proposition \ref{REF:s-sets}, the set of $t$-sets $\mcal{S}_t(\lambda)$ of self-conjugate $t$-cores $\lambda$, i.e. the subset of points $(a_{t,i})_{i\in\ZZ/t\ZZ}\in \mcal{C}_t$ satisfying the additional symmetry $a_{t,i} + a_{t,-1-i} = t-1$ from Corollary \ref{COR:conjugate-s-set}.

\item In $z$-coordinates, the subset of points $(z_{t,i})_{i\in\ZZ/t\ZZ} \in \mcal{C}_t$ satisfying the additional symmetry $z_{t,i} = z_{t,-i}$.

\item In $u$-coordinates, the set of lattice points $\{(u_0,u_1,\ldots,u_{t'})\in \ZZ^{1+t'}: u_0+u_1+\cdots+u_{t'} = s'\}$, with the isomorphism defined by $u_0 \colonequals \floor{z_{t,0}/2} = s' - \sum_{i=1}^{t'} u_i$ and the following additional relations: $u_i \colonequals z_{t,i}$ (for $1\le i\le t'$) if $t$ is odd; and $u_{t'} \colonequals z_{t,t'}/2$ and $u_i \colonequals z_{t,i}$ (for $1\le i\le t' - 1$) if $t$ is even.
\end{itemize}
\end{proposition}

\begin{proof}[Proof of $a$-to-$z$ translation]

As usual, let $k = \frac12(s+1)(t-1)$. Define the map $\phi$ as in Proposition \ref{PROP:change-to-z-coordinates}, so for any $i\in\ZZ/t\ZZ$, the difference $t\cdot(y_i - y_{-i})$ evaluates to
\[
(x_{si+k} - x_{si+s+k} + s) - (x_{-si+k} - x_{-si+s+k} + s)
= (x_{si+k} + x_{-si+s+k}) - (x_{si+s+k} + x_{-si+k}).
\]
Observe that $(si+k)+(-si+s+k) = (-si+k)+(si+s+k) = s+2k \equiv -1\pmod{t}$ for all $i\in\ZZ/t\ZZ$. It follows (since $s$ is coprime to $t$) that $y_i = y_{-i}$ for all $i\in\ZZ/t\ZZ$ if and only if $x_i + x_{-1-i}$ is constant over $i\in\ZZ/t\ZZ$; if and only if $x_i + x_{-1-i} = t-1$ for all $i\in\ZZ/t\ZZ$ (because we are working in the plane $\sum_{i\in\ZZ/t\ZZ} (x_i + x_{-1-i}) = t(t-1)$).
\end{proof}

\begin{proof}[Proof of $z$-to-$u$ translation]

If $y_i = y_{-i}$ for all $i\in\ZZ/t\ZZ$, then $0 y_0 = 0$ and $i y_i + (t-i) y_{t-i} \equiv0\pmod{t}$ for all $i\in\ZZ/t\ZZ$.

If $t$ is odd, then we automatically get $\sum_{i\in\ZZ/t\ZZ} i y_i \equiv 0\pmod{t}$, clearly establishing $U_t(s) = \{(u_0,u_1,\ldots,u_{t'})\in \ZZ_{\ge0}^{1+t'}: u_0+u_1+\cdots+u_{t'} = s'\}$ under the canonical map given by $u_i \colonequals z_{t,i}$ for $1\le i\le t'$ and $u_0 \colonequals \floor{z_{t,0}/2} = s' - \sum_{i=1}^{t'} u_i$ (from $\sum_{i\in\ZZ/t\ZZ} z_{t,i} = s$).

If $t$ is even, then the vanishing of $\sum_{i\in\ZZ/t\ZZ} i y_i \equiv (t/2) y_{t/2}\pmod{t}$ modulo $t$ is equivalent to $y_{t'} \equiv 0\pmod{2}$. This establishes the isomorphism $U_t(s) = \{(u_0,u_1,\ldots,u_{t'})\in \ZZ_{\ge0}^{1+t'}: u_0+u_1+\cdots+u_{t'} = s'\}$ under the canonical map given by $u_{t'} = z_{t,t'} / 2$; $u_i \colonequals z_{t,i}$ for $1\le i\le t' - 1$; and and $u_0 \colonequals \floor{z_{t,0}/2} = s' - \sum_{i=1}^{t'} u_i$ (from $\sum_{i\in\ZZ/t\ZZ} z_{t,i} = s$).

The previous two paragraphs prove the result for $\mcal{D}_s\cap\mcal{D}_t$. The same $z$-to-$u$ conversion works for $\mcal{D}_t$ as well, except one ignores the inequality conditions $u_i \ge 0$ (which come from the conditions $z_{t,i} \ge 0$ of $\map{TD}_t(s)$).
\end{proof}

\begin{markednewpar}\leavevmode\end{markednewpar}\begin{arxiv}
\begin{sidermk}

The simplicity of $U_t(s)$ may seem somewhat mysterious, but note that regardless of the choice of the constant $k$ when defining $z$-coordinates, it should still be intuitively clear that some condition like $y_i = y_{\alpha - i}$ should hold identically for some constant $\alpha$ (depending on $k$), in which case (at least for $t$ odd) we have $\sum_{i\in\ZZ/t\ZZ} i y_i \equiv \frac{t}{2} \alpha\sum_{i\in\ZZ/t\ZZ} y_i = \frac{\alpha s t}{2}\pmod{t}$ independent of the lattice point $(y_i)_{i\in\ZZ/t\ZZ}$ (as long as $\sum_{i\in\ZZ/t\ZZ} y_i = s$).
\end{sidermk}

\begin{sidecor}[Weak version of Ford--Mai--Sze's theorem \cite{FMS}]
\label{COR:SIDE:finitely-many-s,t-cores-self-conjugate}

$B_t(s)$ and $U_t(s)$ are discrete bounded sets, hence finite. In particular, there are finitely many self-conjugate $(s,t)$-cores, so the sums in Theorems \ref{THM:self-conjugate-Armstrong} and \ref{THM:self-conjugate-Fayers} are finite and well-defined.
\end{sidecor}

\begin{sidermk}

Ford--Mai--Sze \cite{FMS} actually computes the exact number of self-conjugate $(s,t)$-cores as $\binom{\floor{s/2}+\floor{t/2}}{\floor{s/2}}$, by bijecting to lattice paths. At the beginning of the proof of Theorem \ref{THM:self-conjugate-Armstrong} we implicitly give Johnson's variant using the $u$-coordinates.
\end{sidermk}
\end{arxiv}

\subsubsection{Size of the stabilizer of a self-conjugate \texorpdfstring{$s$}{s}-core}
\label{SEC:stab-computation-self-conjugate}

The main new observation for Theorem \ref{THM:self-conjugate-Fayers} is the following computational simplification of Fayers' formula for the size of the stabilizer of a self-conjugate $s$-core, based on Proposition \ref{PROP:s-set-t-set-interaction}. This is the self-conjugate analog of Corollary \ref{COR:s-set-and-stabilizer-data-in-z-coordinates} to Proposition \ref{PROP:stab-computation}.

\begin{proposition}[c.f. Fayers {\cite[Proposition 4.9]{Fayers}}]
\label{PROP:stab-computation-self-conjugate}

Fix coprime $s,t\ge1$, and $\lambda\in \mcal{D}_s$. For convenience, let $t' = \floor{t/2}$.

\begin{enumerate}
\item If $t$ is odd, then $\card{\stab_{H_{s,t}}(\lambda)}$ equals
\[ 2^{\floor{z_{t,0}(\lambda^t) / 2}} \floor{z_{t,0}(\lambda^t)/2}! \prod_{i=1}^{t'} z_{t,i}(\lambda^t)!
= 2^{u_0(\lambda^t)} \prod_{i=0}^{t'} u_i(\lambda^t)!. \]

\item If $t$ is even (so $s$ is odd), then $\card{\stab_{H_{s,t}}(\lambda)}$ equals
\[ 2^{\floor{z_{t,0}(\lambda^t) / 2}} \floor{z_{t,0}(\lambda^t)/2}! 2^{\floor{z_{t,t'}(\lambda^t) / 2}} \floor{z_{t,t'}(\lambda^t)/2}! \prod_{i=1}^{t' - 1} z_{t,i}(\lambda^t)!
= 2^{u_0(\lambda^t) + u_{t'}(\lambda^t)}  \prod_{i=0}^{t'} u_i(\lambda^t)!. \]
\end{enumerate}

\end{proposition}

\begin{proof}

Fix $\lambda\in \mcal{D}_s$. Fayers uses Proposition \ref{PROP:extending-bijections-of-s-sets-to-group-actions-self-conjugate} to compute $\card{\stab_{H_{s,t}}(\lambda)}$ as the number of permutations $\pi$ of the symmetric $s$-set $\mcal{S}_s(\lambda)$ with $\pi(s-1-m) = s-1-\pi(m)$ and $\pi(m)\equiv m\pmod{t}$ for all $m \in \mcal{S}_s(\lambda)$, i.e. the product of the following individual contributions:
\begin{itemize}
\item $\card{\mcal{S}_s(\lambda)\cap (j+t\ZZ)}! = \card{\mcal{S}_s(\lambda)\cap (s-1-j+t\ZZ)}!$ for each \emph{pair} of \emph{distinct} residues $\{j,s-1-j\}\pmod{t}$;

\item $2^{\floor{\frac12\card{\mcal{S}_s(\lambda)\cap (j+t\ZZ)}}} \floor{\frac12\card{\mcal{S}_s(\lambda)\cap (j+t\ZZ)}}!$ for each residue $j\pmod{t}$ with $j \equiv s-1-j\pmod{t}$.
\end{itemize}
Fayers evaluates the product in his own way, but it will be easier for us to directly translate to $z$-coordinates via
\[
z_{t,i}(\lambda^t) = \tfrac{1}{t}(a_{t,si + k}(\lambda^t) - [a_{t,s(i+1) + k}(\lambda^t) - s]) = \card{\mcal{S}_s(\lambda)\cap(si+s+k+t\ZZ)}
\]
from Corollary \ref{COR:s-set-and-stabilizer-data-in-z-coordinates} (where $k = \frac12(s+1)(t-1)$ as usual), and then to $u$-coordinates via Proposition \ref{PROP:change-to-z-coordinates-self-conjugate} (as we have $\lambda^t\in \mcal{D}_s\cap\mcal{D}_t$ by Proposition \ref{PROP:extending-Olsson-theorem-self-conjugate}). Observe that $2s+2k \equiv s-1\pmod{t}$.

\begin{enumerate}
\item If $t$ is odd, then the only residue $j$ with $j \equiv s-1-j\pmod{t}$ is $2^{-1}(s-1)\pmod{t}$. It follows that the residues $si+s+k+t\ZZ$, for $0\le i\le \frac{t-1}{2} = t'$, form a set of representatives of the pairs of residues $\{j,s-1-j\}\pmod{t}$, with $i=0$ corresponding to $j = s+k \equiv 2^{-1}(s-1)\pmod{t}$. Thus $\card{\stab_{H_{s,t}}(\lambda)}$ equals the product $2^{\floor{z_{t,0}(\lambda^t) / 2}} \floor{z_{t,0}(\lambda^t)/2}! \prod_{i=1}^{t'} z_{t,i}(\lambda^t)!$.

\item If $t$ is even (so $s$ is odd), then the residues $j$ with $j \equiv s-1-j\pmod{t}$ are $\frac{s-1}{2}+t\ZZ$ and $\frac{s-1}{2}+\frac{t}{2}+t\ZZ$. It follows that the residues $si+s+k+t\ZZ$, for $0\le i\le \frac{t}{2} = t'$, form a set of representatives of the pairs of residues $\{j,s-1-j\}\pmod{t}$, with $i=0$ corresponding to $j = s+k = s+(t-1)\frac{s+1}{2} \equiv \frac{s-1}{2}\pmod{t}$ and $i = \frac{t}{2} = t'$ corresponding to $j = s\frac{t}{2} + s+k \equiv 1\frac{t}{2}+\frac{s-1}{2} = \frac{s-1}{2}+\frac{t}{2}\pmod{t}$. Thus $\card{\stab_{H_{s,t}}(\lambda)}$ equals the product $2^{\floor{z_{t,0}(\lambda^t) / 2}} \floor{z_{t,0}(\lambda^t)/2}! 2^{\floor{z_{t,t'}(\lambda^t) / 2}} \floor{z_{t,t'}(\lambda^t)/2}! \prod_{i=1}^{t' - 1} z_{t,i}(\lambda^t)!$.
\end{enumerate}
Finally, in the odd $t$ case, we can translate to $u$-coordinates using $u_0 = \floor{z_{t,0}/2}$ and $u_i = z_{t,i}$ for $1\le i\le t'$. In the even case, we instead use $z_{t,t'} = 2u_{t'}$, $u_0 = \floor{z_{t,0}/2}$, and $u_i = z_{t,i}$ for $1\le i\le t'-1$.
\end{proof}

\begin{markednewpar}\leavevmode\end{markednewpar}\begin{arxiv}
\begin{sidermk}
We will only explicitly use this result for self-conjugate $(s,t)$-cores $\lambda\in \mcal{D}_s\cap \mcal{D}_t$, when $s,t$ are coprime, in the proof Theorem \ref{THM:self-conjugate-Fayers}.
\end{sidermk}
\end{arxiv}

\subsubsection{Evaluating quadratic sums in \texorpdfstring{$u$}{u}-coordinates}

We will evaluate several special $u$-coordinate sums in our proofs of Theorems \ref{THM:self-conjugate-Armstrong} and \ref{THM:self-conjugate-Fayers} using the following self-conjugate analog of Corollary \ref{COR:special-cyclic-sums}.

\begin{proposition}
\label{PROP:special-u-coordinate-sums}

Fix coprime $s,t\ge1$. Define $U_t(s)$ as in Proposition \ref{PROP:change-to-z-coordinates-self-conjugate} (we will freely switch between $z$-coordinates and $u$-coordinates parameterizing self-conjugate $(s,t)$-cores as appropriate). For convenience, let $s' = \floor{s/2}$ and $t' = \floor{t/2}$, and let $[s]_2 \colonequals s - 2s'\in\{0,1\}$. Then we evaluate the sum $\sum_{(z_{t,j})\in U_t(s)} f(z_{t,0},\ldots,z_{t,t-1})$ in the cases listed below, where the most important terms have been boxed.

First we look at `modified exponential' cases for \emph{odd} $t$, freely using $z_{t,0} = 2u_0 + [s]_2$.
\begin{enumerate}
\item $(1\cdot 2^{-1} + t'\cdot 1)^{s'} = \fbox{$(\frac{t}{2})^{s'}$}$ when ($t$ is odd and) $f = 2^{-u_0}\binom{\abs{\bd{u}}}{\bd{u}}\cdot \fbox{$1$}$, where $\bd{u} = (u_0,\ldots,u_{t'})$;

\item \fbox{$s'(\frac{t}{2})^{s' - 1}$} when $f = 2^{-u_0}\binom{\abs{\bd{u}}}{\bd{u}}\cdot \fbox{$u_i$}$ for some $1\le i\le t'$;

\item \fbox{$s'(s' - 1)(\frac{t}{2})^{s' - 2}$} when either $f = 2^{-u_0}\binom{\abs{\bd{u}}}{\bd{u}}\cdot \fbox{$u_i(u_i - 1)$}$ for some $1\le i\le t'$ \emph{or} $f = 2^{-u_0}\binom{\abs{\bd{u}}}{\bd{u}}\cdot \fbox{$u_i u_j$}$ for some distinct $1\le i<j\le t'$;

\item \fbox{$[s]_2^2(\frac{t}{2})^{s'} + (2+2[s]_2)s'(\frac{t}{2})^{s'-1} + s'(s'-1)(\frac{t}{2})^{s'-2}$} (comes from $[s]_2^2(\frac{t}{2})^{s'} + (4+4[s]_2)\cdot 2^{-1}s'(\frac{t}{2})^{s'-1} + 4\cdot 2^{-2}s'(s'-1)(\frac{t}{2})^{s'-2}$) when $f = 2^{-u_0}\binom{\abs{\bd{u}}}{\bd{u}}\cdot \fbox{$z_{t,0}^2$} = 2^{-u_0}\binom{\abs{\bd{u}}}{\bd{u}}\cdot ([s]_2^2 + (4+4[s]_2)u_0 + 4u_0(u_0 - 1))$ for some $1\le i\le t'$;

\item \fbox{$[s]_2 s'(\frac{t}{2})^{s'-1} + s'(s'-1)(\frac{t}{2})^{s'-2}$} (comes from $[s]_2 s'(\frac{t}{2})^{s'-1} + 2\cdot 2^{-1}s'(s'-1)(\frac{t}{2})^{s'-2}$) when $f = 2^{-u_0}\binom{\abs{\bd{u}}}{\bd{u}}\cdot \fbox{$z_{t,0} u_i$} = 2^{-u_0}\binom{\abs{\bd{u}}}{\bd{u}}\cdot ([s]_2 u_i + 2u_0 u_i)$ for some $1\le i\le t'$.
\end{enumerate}

Next we look at `modified exponential' cases for \emph{even} $t$, freely using $z_{t,0} = 2u_0 + [s]_2$ and $z_{t,t'} = 2u_{t'}$.

\begin{enumerate}
\item $(1\cdot2^{-1}+(\frac{t}{2} - 1)\cdot 1 + 2^{-1})^{s'} = \fbox{$(\frac{t}{2})^{s'}$}$ when ($t$ is even and) $f = 2^{-u_0 - u_{t'}}\binom{\abs{\bd{u}}}{\bd{u}}\cdot \fbox{$1$}$;

\item \fbox{$s'(\frac{t}{2})^{s' - 1}$} when $f = 2^{-u_0 - u_{t'}}\binom{\abs{\bd{u}}}{\bd{u}}\cdot \fbox{$u_i$}$ for some $1\le i\le t' - 1$;

\item \fbox{$s'(s' - 1)(\frac{t}{2})^{s' - 2}$} when either $f = 2^{-u_0 - u_{t'}}\binom{\abs{\bd{u}}}{\bd{u}}\cdot \fbox{$u_i(u_i - 1)$}$ for some $1\le i\le t' - 1$ \emph{or} $f = 2^{-u_0 - u_{t'}}\binom{\abs{\bd{u}}}{\bd{u}}\cdot \fbox{$u_i u_j$}$ for some distinct $1\le i<j\le t' - 1$;

\item $2\cdot 2^{-1} s'(s' - 1)(\frac{t}{2})^{s' - 2} = \fbox{$s'(s' - 1)(\frac{t}{2})^{s' - 2}$}$ when $f = 2^{-u_0 - u_{t'}}\binom{\abs{\bd{u}}}{\bd{u}}\cdot \fbox{$z_{t,t'} u_i$} = 2\cdot 2^{-u_0 - u_{t'}}\binom{\abs{\bd{u}}}{\bd{u}}\cdot u_{t'} u_i$ for some $1\le i\le t' - 1$;

\item $2\cdot 2^{-1} s'(\frac{t}{2})^{s' - 1} = \fbox{$s'(\frac{t}{2})^{s' - 1}$}$ when $f = 2^{-u_0 - u_{t'}}\binom{\abs{\bd{u}}}{\bd{u}}\cdot \fbox{$z_{t,t'}$} = 2\cdot 2^{-u_0 - u_{t'}}\binom{\abs{\bd{u}}}{\bd{u}}\cdot u_{t'}$;

\item $2^2\cdot 2^{-2} s'(s' - 1)(\frac{t}{2})^{s' - 2} + 2^2\cdot 2^{-1} s'(\frac{t}{2})^{s'-1} = \fbox{$s'(s'-1)(\frac{t}{2})^{s'-2} + 2s'(\frac{t}{2})^{s'-1}$}$ when $f = 2^{-u_0 - u_{t'}}\binom{\abs{\bd{u}}}{\bd{u}}\cdot \fbox{$z_{t,t'}^2$} = 2^2\cdot 2^{-u_0 - u_{t'}}\binom{\abs{\bd{u}}}{\bd{u}}[u_{t'}(u_{t'}-1) + u_{t'}]$;

\item \fbox{$[s]_2^2(\frac{t}{2})^{s'} + (2+2[s]_2)s'(\frac{t}{2})^{s'-1} + s'(s'-1)(\frac{t}{2})^{s'-2}$} (comes from $[s]_2^2(\frac{t}{2})^{s'} + (4+4[s]_2)\cdot 2^{-1}s'(\frac{t}{2})^{s'-1} + 4\cdot 2^{-2}s'(s'-1)(\frac{t}{2})^{s'-2}$) when $f = 2^{-u_0 - u_{t'}}\binom{\abs{\bd{u}}}{\bd{u}}\cdot \fbox{$z_{t,0}^2$} = 2^{-u_0 - u_{t'}}\binom{\abs{\bd{u}}}{\bd{u}}\cdot ([s]_2^2 + (4+4[s]_2)u_0 + 4u_0(u_0 - 1))$ for some $1\le i\le t'$;

\item \fbox{$[s]_2 s'(\frac{t}{2})^{s'-1} + s'(s'-1)(\frac{t}{2})^{s'-2}$} (comes from $[s]_2 s'(\frac{t}{2})^{s'-1} + 2\cdot 2^{-1}s'(s'-1)(\frac{t}{2})^{s'-2}$) when $f = 2^{-u_0 - u_{t'}}\binom{\abs{\bd{u}}}{\bd{u}}\cdot \fbox{$z_{t,0} u_i$} = 2^{-u_0 - u_{t'}}\binom{\abs{\bd{u}}}{\bd{u}}\cdot ([s]_2 u_i + 2u_0 u_i)$ for some $1\le i\le t' - 1$;

\item \fbox{$[s]_2 s'(\frac{t}{2})^{s'-1} + s'(s'-1)(\frac{t}{2})^{s'-2}$} (comes from $[s]_2\cdot 2\cdot2^{-1} s'(\frac{t}{2})^{s'-1} + 2^2\cdot 2^{-2}s'(s'-1)(\frac{t}{2})^{s'-2}$) when $f = 2^{-u_0 - u_{t'}}\binom{\abs{\bd{u}}}{\bd{u}}\cdot \fbox{$z_{t,0} z_{t,t'}$} = 2^{-u_0 - u_{t'}}\binom{\abs{\bd{u}}}{\bd{u}}\cdot ([s]_2\cdot 2u_{t'} + 2u_0\cdot 2u_{t'})$.
\end{enumerate}

Finally we look at `ordinary' cases, freely using $z_{t,0} = 2u_0 + [s]_2$.

\begin{enumerate}
\item \fbox{$\binom{s'+t'}{t'}$} when $f = \fbox{$1$}$;

\item \fbox{$2\binom{s'+t'}{t'+2} + \binom{s'+t'}{t'+1}$} when $f = \fbox{$u_i^2$} = u_i(u_i - 1) + u_i$ for some $1\le i\le t'$;

\item \fbox{$\binom{s'+t'}{t'+2}$} when $f = \fbox{$u_i u_j$}$ for some distinct $1\le i<j\le t'$.

\item \fbox{$[s]_2^2\binom{s'+t'}{t'} + (4+4[s]_2)\binom{s'+t'}{t'+1} + 8\binom{s'+t'}{t'+2}$} when $f = \fbox{$z_{t,0}^2$} = [s]_2^2 + (4+4[s]_2)u_0 + 4u_0(u_0 - 1)$ for some $1\le i\le t'$;

\item \fbox{$[s]_2\binom{s'+t'}{t'+1} + 2\binom{s'+t'}{t'+2}$} when $f = \fbox{$z_{t,0} u_i$} = [s]_2 u_i + 2u_0 u_i$ for some $1\le i\le t'$.

\end{enumerate}
\end{proposition}

\begin{markednewpar}\leavevmode\end{markednewpar}\begin{arxiv}
\begin{sidermk}
Again, we will use these explicit calculations below in the proofs of Theorems \ref{THM:self-conjugate-Armstrong} and \ref{THM:self-conjugate-Fayers}, instead of the indirect approaches taken by Johnson \cite{Johnson}.
\end{sidermk}
\end{arxiv}

\begin{proof}[Proof of ``exponential'' cases]

We can mimic the methods in Corollary \ref{COR:special-cyclic-sums}, using the exponential generating function $\exp(2^{-1}Z_0 T)\prod_{i=1}^{t'} \exp(Z_i T)$ when $t$ is odd, but instead the exponential generating function $\exp(2^{-1}Z_0T)\exp(2^{-1}Z_{t'}T)\prod_{i=1}^{t' - 1}\exp(Z_i T)$ when $t$ is even.
\end{proof}

\begin{proof}[Proof of ``ordinary'' cases]

We can mimic the methods in Corollary \ref{COR:special-cyclic-sums}, using the ordinary generating function $\prod_{i=1}^{t'} (1- Z_i T)^{-1}$.
\end{proof}

\subsubsection{Size of a self-conjugate \texorpdfstring{$t$}{t}-core}
\label{SEC:z-coordinates-partition-size-self-conjugate}

Lemma \ref{LEM:size-of-partition-z-coordinates} (for $\mcal{C}_t$) simplifies in $u$-coordinates for $\mcal{D}_t$.

\begin{lemma}
\label{LEM:size-of-partition-z-coordinates-self-conjugate}

Fix coprime $s,t\ge1$ and $\lambda\in \mcal{D}_t$. For convenience, set $s' \colonequals \floor{s/2}$ and $t' \colonequals \floor{t/2}$. We divide into cases based on the parity of $t$, but in both cases, the non-constant (i.e. homogeneous quadratic) part of the quadratic, namely the polynomial given by $\card{\lambda}+\frac{1}{24}(t^2-1)$, has coefficients summing to $0$.

\begin{itemize}
\item If $t$ is odd, then the size $\card{\lambda}$ is given by a quadratic polynomial
\[
-\tfrac{1}{24}(t^2-1) + M_2(u_1,\ldots,u_{t'}) + S_2(u_1,\ldots,u_{t'}) + z_{t,0} L_1(u_1,\ldots,u_{t'}) + \tfrac{1}{24}(t^2-1) z_{t,0}^2
\]
in $z_{t,0},u_1,\ldots,u_{t'}$, where $S_2$ is a homogeneous quadratic with only `square' terms, and with coefficient sum $\frac{1}{24}(t-2)(t^2-1)$; $L_1$ is a homogeneous linear polynomial with coefficient sum $-2\cdot\frac{1}{24}(t^2-1)$; and $M_2$ is the ``leftover'' homogeneous quadratic with only `mixed' terms, and with coefficient sum $-\frac{1}{24}(t-3)(t^2-1)$.

\item If $t$ is even, then the size $\card{\lambda}$ is given by a quadratic polynomial
\begin{align*}
-\tfrac{1}{24}(t^2-1)
&+ M_2(u_1,\ldots,u_{t' - 1})
+ S_2(u_1,\ldots,u_{t' - 1}) + \tfrac{1}{24}(t^2-1)(z_{t,t'}^2 + z_{t,0}^2) \\
&- \tfrac{1}{24}(t^2+2) z_{t,0}z_{t,t'} + z_{t,0} L_0(u_1,\ldots,u_{t' - 1}) + z_{t,t'} L_{t'}(u_1,\ldots,u_{t'-1})
\end{align*}
in $z_{t,0},u_1,\ldots,u_{t' - 1},z_{t,t'}$, where $S_2$ is a homogeneous quadratic with only `square' terms, and with coefficient sum $\frac{1}{24}(t-2)(t^2-2t-2)$; $L_0$ is a homogeneous linear polynomial with coefficient sum $-2\cdot \frac{1}{24}(t^2-1)+\frac{1}{24}(t^2+2) = -\frac{1}{24}(t^2-4)$; $L_{t'}$ is a homogeneous linear polynomial with coefficient sum $2\cdot\frac{1}{24}(t^2+2)$; and $M_2$ is the ``leftover'' homogeneous quadratic with with only `mixed' terms, and with coefficient sum $-\frac{1}{24}(t^3-2t^2+2t+8)$.
\end{itemize}
\end{lemma}

\begin{markednewpar}\leavevmode\end{markednewpar}\begin{arxiv}
\begin{sidermk}
We will only explicitly use this result for $(s,t)$-cores $\lambda\in \mcal{D}_s\cap \mcal{D}_t$, when $s,t$ are coprime, in the proof of Theorems \ref{THM:self-conjugate-Armstrong} and \ref{THM:self-conjugate-Fayers}. But the general $t$-core version may help for other problems.
\end{sidermk}

\begin{sidermk}
In principle we could compute the coefficients more explicitly, but by the computations in Proposition \ref{PROP:special-u-coordinate-sums}, we will not need to.
\end{sidermk}
\end{arxiv}

\begin{proof}[Recap of $z$-coordinate preliminaries]

Write $\card{\lambda} = -\frac{1}{24}(t^2-1) + \frac{1}{2t}\cdot P$ for convenience. Since $\lambda\in \mcal{D}_t\subseteq \mcal{C}_t$ and $s,t$ are coprime, Equation \eqref{EQ:size-a-to-z-translation} (from the proof of Lemma \ref{LEM:size-of-partition-z-coordinates}) gives the $a$-to-$z$ translation $P \colonequals \sum_{i=0}^{t-1} [a_{t,i} - \frac{t-1}{2}]^2 = \sum_{\ell\in\ZZ/t\ZZ}(\sum_{j=0}^{t-1} (\frac{t-1}{2} - j)z_{t,j+\ell})^2$, with $P(1,\ldots,1) = 0$ in the $z$-coordinates. Recall also from Lemma \ref{LEM:size-of-partition-z-coordinates} that $[z_{t,i}^2]P = \sum_{j=0}^{t-1}(\frac{t-1}{2} - j)^2 = \frac{1}{12}t(t^2-1)$ for all $i\in\ZZ/t\ZZ$.
\end{proof}

We finish by breaking into cases based on parity of $t$ and using Proposition \ref{PROP:change-to-z-coordinates-self-conjugate} (which applies to $\mcal{D}_t$, not just $\mcal{D}_s\cap \mcal{D}_t$). For convenience, let $[x]_t\in\{0,1,\ldots,t-1\}$ denote the least nonnegative residue of $x\pmod{t}$.

\begin{proof}[Proof of Lemma \ref{LEM:size-of-partition-z-coordinates-self-conjugate} for odd $t$]

In the $u$-coordinates (using $u_i = z_{t,i} = z_{t,-i}$ for $1\le i\le \frac{t-1}{2} = t'$), the polynomial $P$ takes the form
\[
P = M(u_1,\ldots,u_{t'}) + \sum_{i=1}^{t'} \alpha_i u_i^2 + \sum_{i=1}^{t'} \beta_i u_i z_{t,0} + \alpha_0 z_{t,0}^2. 
\]

We make the following calculations, often suppressing $z_i \colonequals z_{t,i}$.
\begin{itemize}
\item We have $\alpha_0 = [z_0^2]P$ equal to $\sum_{j=0}^{t-1} (\frac{t-1}{2} - j)^2$ (which simplifies as $\frac{1}{12}t(t^2-1)$).

\item To determine $\alpha_i$ (for $1\le i\le \frac{t-1}{2}$), we evaluate $\alpha_i = [u_i^2]P = ([z_i^2]P + [z_{-i}^2]P + [z_i z_{-i}]P)$ as $2\cdot\frac{1}{12}t(t^2-1) + 2\sum_{\ell\in\ZZ/t\ZZ}(\frac{t-1}{2} - [i-\ell]_t) (\frac{t-1}{2} - [-i-\ell]_t)$, or $2\cdot\frac{1}{12}t(t^2-1) + 2\sum_{\ell\in\ZZ/t\ZZ}(\frac{t-1}{2} - [2i+\ell]_t) (\frac{t-1}{2} - [\ell]_t)$. It follows that $\sum_{i=0}^{t'} \alpha_i = [z_0^2]P + \frac12\sum_{i=1}^{t-1} ([z_i^2]P + [z_{-i}^2]P + [z_i z_{-i}]P)$ evaluates to
\[
(t-1)\cdot \frac{t(t^2-1)}{12} + \sum_{\ell\in\ZZ/t\ZZ}\sum_{i=0}^{t-1}\left(\frac{t-1}{2} - [\ell]_t \right) \left(\frac{t-1}{2} - [2i+\ell]_t \right).
\]
But $\gcd(2,t) = 1$, so $\sum_{i=0}^{t-1}(\frac{t-1}{2} - [2i+\ell]_t) = 0$ for each $\ell\in\ZZ/t\ZZ$. Thus $\sum_{i=0}^{t'} \alpha_i = (t-1)\cdot \frac{1}{12}t(t^2-1)$, and $\sum_{i=1}^{t'} \alpha_i = (t-2)\cdot\frac{1}{12}t(t^2-1)$.

\item Similarly, we expand (for $1\le i\le \frac{t-1}{2}$) $\beta_i = ([z_0 z_i] + [z_0z_{-i}])P = 4\sum_{\ell\in\ZZ/t\ZZ} (\frac{t-1}{2} - [i+\ell]_t)(\frac{t-1}{2} - [\ell]_t)$, so that $\sum_{i=1}^{t'} \beta_i = 2\sum_{\ell\in\ZZ/t\ZZ}(\frac{t-1}{2} - [\ell]_t)\sum_{i=1}^{t-1}(\frac{t-1}{2} - [i+\ell]_t) = 2\sum_{\ell}-(\frac{t-1}{2} - [\ell]_t)^2 = -2\cdot \frac{1}{12}t(t^2-1)$.
\end{itemize}
Finally, $P(1,\ldots,1) = 0$ gives $M(1,\ldots,1) = -(t-3)\cdot \frac{1}{12}t(t^2-1)$, and substituting $P$ into $\card{\lambda} = -\frac{1}{24}(t^2-1) +  \frac{1}{2t} P$ finishes the job.
\end{proof}

\begin{proof}[Proof of Lemma \ref{LEM:size-of-partition-z-coordinates-self-conjugate} for even $t$]

In the $u$-coordinates (using $u_i = z_{t,i} = z_{t,-i}$ for $1\le i\le \frac{t}{2} - 1 = t' - 1$), the polynomial $P$ takes the form
\[
P = M(u_1,\ldots,u_{t' - 1}) + \sum_{i=1}^{t' - 1} \gamma_i u_i z_{t,t'} + \sum_{i=1}^{t' - 1} \alpha_i u_i^2 + \sum_{i=1}^{t' - 1} \beta_i u_i z_{t,0} + \alpha_0 z_{t,0}^2 + \beta_{t'} z_{t,0} z_{t'} + \alpha_{t'} z_{t,t'}^2. 
\]
We make the following calculations, often suppressing $z_i \colonequals z_{t,i}$.
\begin{itemize}
\item We have $\alpha_0 = [z_{t,0}^2]P$ and $\alpha_{t'} = [z_{t,t'}^2]P$ both equal to $\sum_{j=0}^{t-1} (\frac{t-1}{2} - j)^2$ (which simplifies as $\frac{1}{12}t(t^2-1)$).

\item For $1\le i\le t' - 1$, we evaluate $\alpha_i = [u_i^2]P = 2\cdot\frac{1}{12}t(t^2-1) + 2\sum_{\ell\in\ZZ/t\ZZ}(\frac{t-1}{2} - [\ell]_t)(\frac{t-1}{2} - [2i+\ell]_t)$ as in the odd $t$ case. It follows that $\sum_{i=0}^{t'} \alpha_i = [z_0^2]P + [z_{t'}^2]P + \frac12\sum_{0<\abs{i-t'}<t'} ([z_i^2]P + [z_{-i}^2]P + [z_i z_{-i}]P)$ evaluates to
\[ (t-2)\cdot\frac{t(t^2-1)}{12} + \sum_{\ell\in\ZZ/t\ZZ}\left(\frac{t-1}{2} - [\ell]_t \right)\sum_{i=0}^{t-1} \left(\frac{t-1}{2} - [2i+\ell]_t \right). \]
Note that $\sum_{i=0}^{t-1} (\frac{t-1}{2} - [2i+\ell]_t)$ depends only on the parity of the residue class $\ell\pmod{t}$; it equals $t\frac{t-1}{2} - 2(0+2+\cdots+(t-2)) = (1+3+\cdots+(t-1)) - (0+2+\cdots+(t-2)) = t'$ if $\ell$ is even, and $-t'$ if $\ell$ is odd. Thus $\sum_{\ell\in\ZZ/t\ZZ}(\frac{t-1}{2} - [\ell]_t)\sum_{i=0}^{t-1} (\frac{t-1}{2} - [2i+\ell]_t)$ evaluates to $\sum_{j=0}^{t' - 1}(\frac{t-1}{2} - 2j)t' + (\frac{t-1}{2} - 2j-1)(-t') = \sum_{j=0}^{t'-1}1\cdot t' = \frac{1}{4}t^2$, whence $\sum_{i=0}^{t'} \alpha_i = (t-2)\cdot \frac{1}{12}t(t^2-1) + \frac{1}{4}t^2$, and $\sum_{i=1}^{t' - 1} \alpha_i = (t-4)\cdot \frac{1}{12}t(t^2-1) + \frac{1}{4}t^2 = \frac{1}{12}t(t-2)(t^2-2t-2)$.

\item Next, $\beta_{t'} = [z_0 z_{t'}] P = 2\sum_{\ell\in\ZZ/t\ZZ} (\frac{t-1}{2} - [t' +\ell]_t)(\frac{t-1}{2} - [\ell]_t)$, which simplifies to $2\sum_{j=0}^{t'-1}(\frac{t-1}{2} - j)(\frac{t-1}{2}-j-t')+(\frac{t-1}{2} -j-t')(\frac{t-1}{2} -j) = -\frac{1}{3}t'(2t'^2+1) = - \frac{1}{12}t(t^2+2)$.

\item Similarly, we expand (for $1\le i\le t' - 1$) $\beta_i = ([z_0 z_i] + [z_0z_{-i}])P = 4\sum_{\ell\in\ZZ/t\ZZ} (\frac{t-1}{2} - [i+\ell]_t)(\frac{t-1}{2} - [\ell]_t)$, so that $\sum_{i=1}^{t'} \beta_i = 2\sum_{\ell}(\frac{t-1}{2} - [\ell]_t)\sum_{i=1}^{t-1}(\frac{t-1}{2} - [i+\ell]_t) = 2\sum_{\ell}-(\frac{t-1}{2} - [\ell]_t)^2 = -2\cdot \frac{1}{12}t(t^2-1)$. It follows that $\sum_{i=1}^{t' - 1} \beta_i = -2\cdot \frac{1}{12}t(t^2-1) - \beta_{t'} = -2\cdot \frac{1}{12}t(t^2-1) + \frac{1}{12}t(t^2+2) = -\frac{1}{12}t(t^2-4)$.

\item It remains to compute, for $1\le i\le t' - 1$, the coefficient $\gamma_i = ([z_{t'}z_i]+[z_{t'}z_{-i}])P = 4\sum_{\ell\in\ZZ/t\ZZ}(\frac{t-1}{2} - [t'+i+\ell]_t)(\frac{t-1}{2} - [\ell]_t)$, so that $\beta_{t'} + \sum_{i=1}^{t'-1}\gamma_i = 2\sum_{\ell\in\ZZ/t\ZZ} (\frac{t-1}{2} - [\ell]_t)\sum_{i=1}^{t-1}(\frac{t-1}{2} - [t'+i+\ell]_t) = 2\sum_{\ell\in\ZZ/t\ZZ}-(\frac{t-1}{2}-[\ell]_t)(\frac{t-1}{2}-[t'+\ell]_t) = -\beta_{t'}$. Thus $\sum_{i=1}^{t'-1}\gamma_i = -2\beta_{t'} = 2\cdot \frac{1}{12}t(t^2+2)$.
\end{itemize}
Finally, $P(1,\ldots,1) = 0$ gives $M(1,\ldots,1) = -\frac{1}{12}t(t^3-2t^2+2t+8)$, and substituting $P$ into $\card{\lambda} = -\frac{1}{24}(t^2-1) +  \frac{1}{2t} P$ finishes the job.
\end{proof}

\subsection{Proofs of self-conjugate conjectures}
\label{SEC:main-theorem-proofs-self-conjugate}

This section is the self-conjugate analog of Section \ref{SEC:main-theorem-proofs}, addressing Armstrong's and Fayers' self-conjugate conjectures.

\begin{proof}[Proof of Theorem \ref{THM:self-conjugate-Armstrong} by direct computation for odd $t$]
We use $u$-coordinates. For convenience, let $s' = \floor{s/2}$ and $t' = \floor{t/2}$, and let $[s]_2 \colonequals s - 2s'\in\{0,1\}$. By Proposition \ref{PROP:special-u-coordinate-sums}, the denominator is just $\sum_{U_t(s)}1 = \binom{t'+s'}{t'}$.

For odd $t$, Lemma \ref{LEM:size-of-partition-z-coordinates-self-conjugate} says the size $\card{\lambda}$ is given by a quadratic polynomial
\[
-\tfrac{1}{24}(t^2-1) + M_2(u_1,\ldots,u_{t'}) + S_2(u_1,\ldots,u_{t'}) + z_{t,0} L_1(u_1,\ldots,u_{t'}) + \tfrac{1}{24}(t^2-1) z_{t,0}^2
\]
in $z_{t,0},u_1,\ldots,u_{t'}$, where $M_2$ is a homogeneous quadratic with only `mixed' terms, and coefficient sum $-\frac{1}{24}(t-3)(t^2-1)$; $S_2$ is a homogeneous quadratic with only `square' terms, and coefficient sum $\frac{1}{24}(t-2)(t^2-1)$; and $L_1$ is a homogeneous linear polynomial with coefficient sum $-2\cdot \frac{1}{24}(t^2-1)$. Then by Proposition \ref{PROP:special-u-coordinate-sums}, the numerator $\sum_{U_t(s)}\card{\lambda}$ evaluates to
\begin{align*}
& -\frac{t^2-1}{24}\cdot \binom{N}{t'} - \frac{(t-3)(t^2-1)}{24}\cdot \binom{N}{t'+2}
+ \frac{(t-2)(t^2-1)}{24} \left[ 2\binom{N}{t'+2} + \binom{N}{t'+1} \right] \\
&-2\frac{t^2-1}{24} \left[ [s]_2\binom{N}{t'+1} + 2\binom{N}{t'+2} \right] \\
&+ \frac{t^2-1}{24} \left[ [s]_2^2\binom{N}{t'} + (4+4[s]_2)\binom{N}{t'+1} + 8\binom{N}{t'+2} \right] ,
\end{align*}
where we have suppressed $N\colonequals s'+t'$. Collecting terms, and using $\binom{N}{t'+1} = \frac{s'}{t'+1}\binom{N}{t'}$ and $\binom{N}{t'+2} = \frac{s'-1}{t'+2}\binom{N}{t'+1}$ (viewed as polynomial identities in $s'$, for fixed $t'\ge0$), the numerator becomes
\begin{align*}
\frac{t^2-1}{24} &\biggl[ (-1+[s]_2^2)\binom{N}{t'} + [-(t-3)+2(t-2)-2\cdot2+8] \binom{N}{t'+2} \\
&+ [(t-2)-2[s]_2 + (4+4[s]_2)]\binom{N}{t'+1} \biggr] \\
&= \frac{t^2-1}{24}\binom{N}{t'} \left[ (-1+[s]_2^2) + \frac{4s'(s'-1)}{t+1} + [t+2+2[s]_2]\frac{2s'}{t+1} \right],
\end{align*}
where we have used $t'+2 = \frac{t-1}{2}+2 = \frac{t+3}{2}$ to simplify. Factoring $t^2-1 = (t-1)(t+1)$ and multiplying through by $t+1$ transforms the numerator expression to the product of $\frac{1}{24}(t-1)\binom{N}{t'}$ with
\[
(-1+[s]_2^2)(t+1) + 4s'^2+4[s]_2s'+2ts' = -t-1 + (2s'+[s]_2)^2 + ([s]_2^2+2s')t.
\]
But $[s]_2^2 = [s]_2$ (and $2s'+[s]_2 = s$), so the product evaluates to $\frac{1}{24}(t-1)\binom{N}{t'}(-t-1 + s^2 + st)$, which factors as $\frac{1}{24}(t-1)\binom{N}{t'}(s-1)(s+t+1)$. Dividing numerator by denominator yields the desired ratio of $\frac{1}{24}(s-1)(t-1)(s+t+1)$.
\end{proof}

\begin{proof}[Proof of Theorem \ref{THM:self-conjugate-Armstrong} for even $t$]

If $t$ is even, then $s$ is odd. But the un-weighted problem is symmetric in $s,t$, so swapping the roles of $s,t$ in the previous proof establishes the claim.
\end{proof}

\begin{markednewpar}\leavevmode\end{markednewpar}\begin{arxiv}
\begin{sidermk}
In principle we could give a more direct proof for even $t$ along the same lines as the odd $t$ proof, but we do not do so since our main focus is on Fayers' conjectures and not Armstrong's conjectures.
\end{sidermk}
\end{arxiv}

The same techniques prove Fayers' self-conjugate conjecture.

\begin{proof}[Proof of Theorem \ref{THM:self-conjugate-Fayers} for odd $t$]

We use $u$-coordinates. For convenience, let $s' = \floor{s/2}$ and $t' = \floor{t/2}$, and let $[s]_2 \colonequals s - 2s'\in\{0,1\}$. By Proposition \ref{PROP:stab-computation-self-conjugate} and Proposition \ref{PROP:special-u-coordinate-sums}, the denominator times $s'!$ is just $\sum_{U_t(s)}2^{-u_0}\binom{\abs{\bd{u}}}{\bd{u}} = (\frac{t}{2})^{s'}$ where $\bd{u} = (u_0,\ldots,u_{t'})$.

For odd $t$, Lemma \ref{LEM:size-of-partition-z-coordinates-self-conjugate} says the size $\card{\lambda}$ is given by a quadratic polynomial
\[
-\tfrac{1}{24}(t^2-1) + M_2(u_1,\ldots,u_{t'}) + S_2(u_1,\ldots,u_{t'}) + z_{t,0} L_1(u_1,\ldots,u_{t'}) + \tfrac{1}{24}(t^2-1) z_{t,0}^2
\]
in $z_{t,0},u_1,\ldots,u_{t'}$, where $M_2$ is a homogeneous quadratic with only `mixed' terms, and coefficient sum $-\frac{1}{24}(t-3)(t^2-1)$; $S_2$ is a homogeneous quadratic with only `square' terms, and coefficient sum $\frac{1}{24}(t-2)(t^2-1)$; and $L_1$ is a homogeneous linear polynomial with coefficient sum $-2\cdot \frac{1}{24}(t^2-1)$. Then by Propositions \ref{PROP:stab-computation-self-conjugate} and \ref{PROP:special-u-coordinate-sums}, the numerator times $s'!$, i.e. the sum $\sum_{U_t(s)}2^{-u_0}\binom{\abs{\bd{u}}}{\bd{u}}\card{\lambda}$, evaluates to
\begin{align*}
&-\tfrac{1}{24}(t^2-1)\cdot N^{s'}
- \tfrac{1}{24}(t-3)(t^2-1)\cdot s'(s'-1)N^{s'-2} \\
&+ \tfrac{1}{24}(t-2)(t^2-1) [s'(s'-1)N^{s'-2}
+ s' N^{s'-1}] \\
&-2\cdot \tfrac{1}{24}(t^2-1)([s]_2s'N^{s'-1} + s'(s'-1)N^{s'-2}) \\
&+ \tfrac{1}{24}(t^2-1)([s]_2^2 N^{s'} + (2+2[s]_2)s' N^{s'-1} + s'(s'-1)N^{s'-2}),
\end{align*}
where we have suppressed $N\colonequals \frac{t}{2}$. Collecting terms, we have
\begin{itemize}
\item $N^{s'}$ coefficient $-\frac{1}{24}(t^2-1) + \frac{1}{24}(t^2-1)[s]_2^2$;

\item $s'N^{s'-1}$ coefficient $\frac{1}{24}(t^2-1)[(t-2) - 2[s]_2 + (2+2[s]_2)]$, which is just $\frac{1}{24}(t^2-1)\cdot t$;

\item $s'(s'-1)N^{s'-2}$ coefficient $-\frac{1}{24}(t-3)(t^2-1) + \frac{1}{24}(t-2)(t^2-1) - 2\cdot \frac{1}{24}(t^2-1) + \frac{1}{24}(t^2-1)$, or more conceptually, the sum of coefficients of the polynomial $\card{\lambda} + \frac{1}{24}(t^2-1)$, which is $0$ as noted at the case-free beginning of Lemma \ref{LEM:size-of-partition-z-coordinates-self-conjugate},
\end{itemize}
so that the numerator is $N^{s'}\cdot \frac{1}{24}(t^2-1)(-1+[s]_2^2 + t\cdot \frac{s'}{N})$, which simplifies, via $[s]_2^2 = [s]_2 = s - 2s'$, as $N^{s'}\cdot \frac{1}{24}(t^2-1)(s-1)$. Finally, dividing numerator (times $s'!$) by denominator (times $s'!$) yields the desired ratio of $\frac{1}{24}(s-1)(t^2-1)$.
\end{proof}

\begin{proof}[Proof of Theorem \ref{THM:self-conjugate-Fayers} for even $t$]

We use $u$-coordinates. For convenience, let $s' = \floor{s/2}$ and $t' = \floor{t/2}$, and let $[s]_2 \colonequals s - 2s'\in\{0,1\}$. By Proposition \ref{PROP:stab-computation-self-conjugate} and Proposition \ref{PROP:special-u-coordinate-sums}, the denominator times $s'!$ is just $\sum_{U_t(s)}2^{-u_0}\binom{\abs{\bd{u}}}{\bd{u}} = (\frac{t}{2})^{s'}$.

For even $t$, Lemma \ref{LEM:size-of-partition-z-coordinates-self-conjugate} says the size $\card{\lambda}$ is given by a quadratic polynomial
\begin{align*}
-\tfrac{1}{24}(t^2-1)
&+ M_2(u_1,\ldots,u_{t' - 1})
+ S_2(u_1,\ldots,u_{t' - 1}) + \tfrac{1}{24}(t^2-1)(z_{t,t'}^2 + z_{t,0}^2) \\
&- \tfrac{1}{24}(t^2+2) z_{t,0}z_{t,t'} + z_{t,0} L_0(u_1,\ldots,u_{t' - 1}) + z_{t,t'} L_{t'}(u_1,\ldots,u_{t'-1})
\end{align*}
in $z_{t,0},u_1,\ldots,u_{t' - 1},z_{t,t'}$, where $M_2$ is a homogeneous quadratic with with only `mixed' terms, and with coefficient sum $-\frac{1}{24}(t^3-2t^2+2t+8)$; $S_2$ is a homogeneous quadratic with only `square' terms, and coefficient sum $\frac{1}{24}(t-2)(t^2-2t-2)$; $L_0$ is a homogeneous linear polynomial with coefficient sum $-2\cdot \frac{1}{24}(t^2-1)+\frac{1}{24}(t^2+2) = -\frac{1}{24}(t^2-4)$; and $L_{t'}$ is a homogeneous linear polynomial with coefficient sum $2\cdot \frac{1}{24}(t^2+2)$. Then by Propositions \ref{PROP:stab-computation-self-conjugate} and \ref{PROP:special-u-coordinate-sums}, the numerator times $s'!$, i.e. the sum $\sum_{U_t(s)}2^{-u_0}\binom{\abs{\bd{u}}}{\bd{u}}\card{\lambda}$, is just the sum of the following terms:
\begin{itemize}
\item $-\frac{1}{24}(t^2-1) \cdot N^{s'}$ (from constant term);

\item $-\frac{1}{24}(t^3-2t^2+2t+8)\cdot s'(s'-1)N^{s' - 2}$ (from $M_2$);

\item $\frac{1}{24}(t-2)(t^2-2t-2)\cdot [s'(s'-1)N^{s'-2} + s'N^{s'-1}]$ (from $S_2$);

\item $\frac{1}{24}(t^2-1)\cdot [s'(s'-1)N^{s'-2} + 2s'N^{s'-1}]$ (from $z_{t,t'}^2$);

\item $\frac{1}{24}(t^2-1)\cdot [[s]_2^2 N^{s'} + (2+2[s]_2)s' N^{s'-1} + s'(s'-1)N^{s'-2}]$ (from $z_{t,0}^2$);

\item $-\frac{1}{24}(t^2+2)\cdot [[s]_2s' N^{s'-1} + s'(s'-1) N^{s'-2}]$ (from $z_{t,0} z_{t,t'}$);

\item $-\frac{1}{24}(t^2-4)\cdot [[s]_2 s' N^{s'-1} + s'(s'-1) N^{s'-2}]$ (from $z_{t,0} L_0$);

\item $2\cdot \frac{1}{24}(t^2+2)\cdot s'(s' - 1) N^{s' - 2}$ (from $z_{t,t'} L_{t'}$),
\end{itemize}
where we have suppressed $N\colonequals \frac{t}{2}$. Collecting terms, we have
\begin{itemize}
\item $N^{s'}$ coefficient $-\frac{1}{24}(t^2-1) + \frac{1}{24}(t^2-1)[s]_2^2$, which is $0$ since $s$ is odd;

\item $s'N^{s'-1}$ coefficient $\frac{1}{24}(t-2)(t^2-2t-2) + 2\cdot \frac{1}{24}(t^2-1) + \frac{1}{24}(t^2-1)(2+2[s]_2) - [s]_2\frac{1}{24}(t^2+2) - \frac{1}{24}(t^2-4)$, which is $\frac{1}{24}t(t^2+2)$ since $s$ is odd;

\item $s'(s'-1)N^{s'-2}$ coefficient $-\frac{1}{24}(t^3-2t^2+2t+8) + \frac{1}{24}(t-2)(t^2-2t-2) + \frac{1}{24}(t^2-1) + \frac{1}{24}(t^2-1) - \frac{1}{24}(t^2+2) - \frac{1}{24}(t^2-4) + 2\frac{1}{24}(t^2+2)$, or more conceptually, the sum of coefficients of the polynomial $\card{\lambda} + \frac{1}{24}(t^2-1)$, which is $0$ as noted at the case-free beginning of Lemma \ref{LEM:size-of-partition-z-coordinates-self-conjugate}.
\end{itemize}
Thus the numerator times $s'!$ is $s'N^{s'-1}\cdot\frac{1}{24}t(t^2+2) = \frac{s-1}{2}N^{s'}\cdot 2\cdot\frac{1}{24}(t^2+2) = N^{s'}\cdot \frac{1}{24}(s-1)(t^2+2)$, where we have used $s' = \frac{s-1}{2}$ (as $s$ must be odd). Dividing numerator (times $s'!$) by denominator (times $s'!$) yields the desired ratio of $\frac{1}{24}(s-1)(t^2+2)$.
\end{proof}

\section{Counting \texorpdfstring{$(m,m+d,m+2d)$}{(m,m+d,m+2d)}-cores}
\label{SEC:Amdeberhan--Leven--conjecture}

We give two cyclic shifts proofs of Theorem \ref{THM:Amdeberhan--Leven--conjecture}, using different $z$-coordinate-based parameterizations of $\mcal{C}_m\cap \mcal{C}_{m+d}\cap \mcal{C}_{m+2d}$, which may generalize in different ways.

\begin{remark}
Explicitly, the proofs differ in that the first ``symmetric proof'' views $(m,m+d,m+2d)$-cores as $(m+d)$-cores that are also $(m,m+2d)$-cores, while the second ``asymmetric proof'' views $(m,m+d,m+2d)$-cores as $(m,m+d)$-cores that are also $(m+2d)$-cores. The proofs both use (the change of variables from) Proposition \ref{PROP:change-to-z-coordinates} for different choices of coprime $s,t\ge1$, but with $s$ only a ``purely algebraic parameter'' in the first proof.\begin{arxiv} (C.f. Side Remark \ref{RMK:SIDE:z-coordinates-t-core-clarification-of-s}.)\end{arxiv}
\end{remark}

Our first proof uses the extension of Proposition \ref{PROP:change-to-z-coordinates} for general $t$-cores, with $t = m+d$.

\begin{proof}[Symmetric proof]

Let $(s,t) = (d,m+d)$, so $s,t\ge1$ are coprime. By Lemma \ref{LEM:c,a-coordinate-parameterization-s,t-cores}, a $t$-core $\lambda$ lies in $\mcal{C}_{t-d} = \mcal{C}_m$ if and only if $a_{t,i} \ge a_{t,i-d} - (t-d)$ for all $i\in\ZZ/t\ZZ$, and $\lambda\in \mcal{C}_{t+d} = \mcal{C}_{m+2d}$ if and only if $a_{t,i} \ge a_{t,i+d} - (t+d)$ for all $i$. Thus $\lambda\in \mcal{C}_t$ is a $(t-d,t,t+d) = (m,m+d,m+2d)$-core if and only if
\[ t \ge a_{t,i} - [a_{t,i+d} - d] \ge -t \]
for all $i$.\footnote{As described in Remark \ref{RMK:key-quantities-s-set-t-set-interaction}, Johnson essentially observed that the $d$-hooks in $\lambda$ correspond to the indices $i\in\ZZ/t\ZZ$ satisfying $a_{t,i} - [a_{t,i+d} - d] = -t$.} By Proposition \ref{PROP:change-to-z-coordinates} (applied to $(s,t) = (d,m+d)$) parameterizing $\mcal{C}_t = \mcal{C}_{m+d}$, these inequalities translate in $z_t$-coordinates---after division by $t$---to $1 \ge z_{t,j} \ge -1$ (i.e. $z_{t,j}\in\{-1,0,1\}$), along with the usual $\sum_{j\in\ZZ/t\ZZ}z_{t,j} = s = d$; $z_{t,j} \equiv 0\pmod{1}$; and $\sum_{j\in\ZZ/t\ZZ} jz_{t,j} \equiv 0\pmod{t}$.

A cyclic shifts argument analogous to Proposition \ref{PROP:cyclic-shift} yields
\[ \sum_{(z_{t,j})\in \mcal{C}_m\cap \mcal{C}_{m+d}\cap \mcal{C}_{m+2d}} 1 = \frac{1}{t}\sum_{\substack{x_j\in\{-1,0,1\} \\ \sum_{j\in\ZZ/t\ZZ} x_j = d}} 1 = \frac{1}{t} \sum_{i=0}^{\floor{(t-d)/2}} \binom{t}{i,i+d,t-(2i+d)}, \]
where in the last step we count valid sequences $(x_j)_{j\in\ZZ/t\ZZ}\in\{-1,0,1\}^t$ by the numbers $i,i+d,t-(2i+d)$ of appearances of $-1,+1,0$, respectively. Substituting in $t = m+d$ completes the proof. 
\end{proof}

Our second proof uses the more familiar form of Proposition \ref{PROP:change-to-z-coordinates}, for $(s,t)$-cores.

\begin{proof}[Asymmetric proof]

Let $(s,t) = (m+d,m)$, so $s,t\ge1$ are coprime. By Lemma \ref{LEM:c,a-coordinate-parameterization-s,t-cores}, a $t$-core $\lambda$ lies in $\mcal{C}_s = \mcal{C}_{m+d}$ if and only if $a_{t,i} \ge a_{t,i+s} - s$ for all $i\in\ZZ/t\ZZ$, and $\lambda\in \mcal{C}_{s+d} = \mcal{C}_{m+2d} = \mcal{C}_{2s - t}$ if and only if $a_{t,i} \ge a_{t,i+2s} - (2s-t)$ for all $i$. Thus $\lambda\in \mcal{C}_t$ is a $(t,s,2s-t) = (m,m+d,m+2d)$-core if and only if $a_{t,i} - [a_{t,i+s} - s] \ge 0$ and $a_{t,i} - [a_{t,i+2s} - 2s] \ge t$ for all $i$. By Proposition \ref{PROP:change-to-z-coordinates} (applied to $(s,t) = (m+d,m)$) parameterizing $\mcal{C}_s\cap \mcal{C}_t = \mcal{C}_{m+d}\cap \mcal{C}_m$, these inequalities translate in $z_t$-coordinates---after division by $t$---to $z_{t,j}+z_{t,j+1} \ge 1$, along with the usual $z_{t,j}\ge 0$; $\sum_{j\in\ZZ/t\ZZ}z_{t,j} = s = m+d$; $z_{t,j} \equiv 0\pmod{1}$; and $\sum_{j\in\ZZ/t\ZZ} jz_{t,j} \equiv 0\pmod{t}$.

A cyclic shifts argument analogous to Proposition \ref{PROP:cyclic-shift} yields
\[ \sum_{(z_{t,j})\in \mcal{C}_m\cap \mcal{C}_{m+d}\cap \mcal{C}_{m+2d}} 1 = \frac{1}{t}\sum_{\substack{x_j,x_j+x_{j+1}-1 \ge 0 \\ \sum_{j\in\ZZ/t\ZZ} x_j = s}} 1 = \frac{1}{t} \sum_{i=0}^{\floor{t/2}} \frac{t}{t-i}\binom{t-i}{i} \cdot \binom{s-1}{(t-i) - 1}, \]
where in the last step we count valid sequences $(x_j)_{j\in\ZZ/t\ZZ}$ by first choosing the $i$ pairwise non-neighboring $0$ terms, and then the remaining $t-i$ positive integers summing to $s$. There are $\frac{t}{t-i}\binom{t-i}{i}$ ways to choose the $i$ zero-positions (we can double-count pairs $(S,\alpha)$ with $S\subseteq\ZZ/t\ZZ$ the set of $i$ zero-positions, and $\alpha\notin S$ some non-zero-position), and $\binom{s-1}{(t-i)-1}$ ways to choose the $t-i$ positive integers. To finish, we note $\frac{1}{t-i}\binom{t-i}{i}\binom{s-1}{t-i-1} = \frac{1}{s}\binom{s}{i,s-t+i,t-2i}$ and then substitute in $t = m$ and $s = m+d$.
\end{proof}

\begin{remark}
Using either parameterization and mimicking the proof of Theorems \ref{THM:general-Armstrong} (in particular the cyclic shifts), one could evaluate the sum of the sizes of $(m,m+d,m+2d)$-cores. Mimicking Section \ref{SEC:self-conjugate-analogs}, one could also study self-conjugate analogs for $\mcal{D}_m\cap \mcal{D}_{m+d}\cap \mcal{D}_{m+2d}$. Also, in principle, one could probably carry much of this over to $(m,m+d,\ldots,m+kd)$-cores (for $k\ge3$), or even to more general multiple cores, with the computational messiness growing with the complexity of the set of avoided hook lengths. However, it would likely be unenlightening to explicitly work out these directions without new ideas.
\end{remark}

\section{Future work}
\label{SEC:future-work}




This section discusses possibilities for future work.

\begin{question}[Significance of coordinates]
How are the (asymmetric) $z$-coordinates related to other bijections (Anderson \cite{Anderson} for general cores; Ford--Mai--Sze \cite{FMS} for self-conjugate cores; the poset formulation of Stanley--Zanello \cite{SZ} developed further in \cite{Aggarwal}, \cite{AL}, and \cite{Amol}; etc.), which are generally more symmetric in $s$ and $t$?\begin{arxiv} (C.f. Side Remarks \ref{RMK:SIDE:Anderson-vs-Johnson-bijections} and \ref{RMK:SIDE:z-coordinates-significance}.)\end{arxiv}
\end{question}


\begin{question}
Can the computations, especially for Fayers' conjectures (Theorems \ref{THM:general-Fayers} and \ref{THM:self-conjugate-Fayers}) be simplified along the lines of Johnson's weighted Ehrhart reciprocity methods in \cite{Johnson} for Armstrong's conjectures? Is there an `exponential' version of Ehrhart reciprocity?
\end{question}

\begin{remark}
\label{RMK:conceptual-proof}

In Theorems \ref{THM:general-Armstrong} and \ref{THM:general-Fayers} (and with some parity casework in Theorems \ref{THM:self-conjugate-Armstrong} and \ref{THM:self-conjugate-Fayers}) one can certainly use implicit variants of the generating function calculations (which, for Fayers' weighted sums, give an `exponential' prototype of basic Ehrhart \emph{theory}, but not necessarily \emph{reciprocity}) to first show that the average size is a polynomial in $s$ of degree at most $2$, and then give $3$ easy-to-determine pieces of information to uniquely determine the quadratic. One piece of information could be Remark \ref{RMK:conceptual-cancellation-Fayers-general}, i.e. that Fayers' weighted average is linear in $s$; taking $s=1$ or `$s=0$' could be two other pieces of information. For Armstrong's un-weighted average, Johnson \cite{Johnson} uses the $3$ pieces `$s=0$', $s=1$, and $s = -t-1$ (using weighted Ehrhart reciprocity in the last case; one could also more concretely use the vanishing of $\binom{s+t-1}{t+*}$ for $*\ge0$). However, concrete explicit computations have their benefits, and implicit variants are already well-exposited in \cite{Johnson}.
\end{remark}



More concretely, one could study sums of higher powers of $\card{\lambda}$ or other statistics.

\begin{question}
\label{QUES:higher-powers?}

For integers $e\ge0$, what can one say about $\sum_{\lambda\in \mcal{C}_s\cap \mcal{C}_t} \card{\lambda}^e$, or the weighted sum $\sum_{\lambda\in \mcal{C}_s\cap \mcal{C}_t} \card{\stab_{G_{s,t}}(\lambda)}^{-1} \cdot \card{\lambda}^e$? Can we study other statistics, such as length (number of nonzero parts)? Do these have a nice form? What do they say about the distribution of $(s,t)$-cores, indexed by size?
\end{question}

\begin{remark}
As Levent Alpoge points out, it may be more natural to look at \emph{moments} (e.g. in the un-weighted case, essentially sums of powers of $\card{\lambda} - \frac{1}{24}(s-1)(t-1)(s+t+1)$). For specific conjectures of Ekhad and Zeilberger (supported by numerical evidence), see \cite[First Challenge]{Zeilberger} (which was posted on the arXiv a few weeks after v3 of the present paper).
\end{remark}

\begin{remark}
\label{RMK:higher-power-methods}

For any finite exponent $e$ these sums can in principle be computed explicitly by the cyclic shift methods used in Theorems \ref{THM:general-Armstrong} and \ref{THM:general-Fayers}. As a start, the generating function calculations show that both the un-weighted and weighted averages are polynomials in $s$ (but a priori only rational expressions in $t$) of degree at most $2e$. Using the higher degree analog of Remark \ref{RMK:conceptual-cancellation-Fayers-general}, one can reduce the degree bound to $2e-1$ in the weighted case. In the un-weighted case, one can use the symmetry in $s,t$ to show that the un-weighted average is always a polynomial in $s,t$, hence in fact a symmetric polynomial.
\end{remark}

Alternatively, one could ask about generating functions of $(s,t)$-cores. Note that the generating function for $s$-cores (indexed by size) has a well-known nontrivial expression, given for instance in \cite[Bijection 1]{Cranks-t-cores-expos} (based on the $s$-core operation).

\begin{question}
Does the generating function for $(s,t)$-cores (indexed by size) have a nice form or any interesting properties?
\end{question}

\begin{remark}
Of course, if $s,t$ are coprime, this is a polynomial (as there are finitely many $(s,t)$-cores). See \cite{Keith} for some potential progress by W. Keith in this direction. On the other hand, if $g\colonequals\gcd(s,t)>1$, then \cite{AKS} gives the $(s,t)$-generating function in terms of the $(s/g,t/g)$-generating function (thanks to Rishi Nath for pointing out this reference).
\end{remark}

One could perhaps also ask further interesting questions about $t$-cores of different $s$-cores, following \cite{F1,Fayers}. For example, are there unexplored natural definitions of ``randomness''? Or is there anything one can do with the following remark?


\begin{remark}
The group $G_{s,t}$ (Definition \ref{DEF:level-t-group-action}) acts not only on $\mcal{C}_s$ (via $s$-sets or beta-sets of $s$-cores), but on the whole set of partitions (via beta-sets). Similarly, the group $H_{s,t}$ (Definition \ref{DEF:level-t-group-action-self-conjugate}) acts not only on $\mcal{D}_s$, but also on $\mcal{C}_s$ (via $s$-sets or beta-sets), and the whole set of partitions (via beta-sets).
\end{remark}

In a different direction, one could investigate numerical semigroups containing $s,t$, which inject into $\mcal{C}_s\cap\mcal{C}_t$ (as mentioned in the introduction). Do the techniques for $(s,t)$-cores help?

\subsection*{Acknowledgements}

This research was conducted at the University of Minnesota Duluth REU and was supported by NSF grant 1358695 and NSA grant H98230-13-1-0273. The author especially thanks Joe Gallian for suggesting the problem and carefully proofreading the paper, Melanie Wood for suggesting exactly what to emphasize or de-emphasize, and Amol Aggarwal, Levent Alpoge, Ben Gunby, Rishi Nath, and Paul Johnson for interesting discussions and other helpful comments on the manuscript. The author also thanks Rishi Nath for pointing out the papers \cite{AL} and \cite{AKS}, Tewodros Amdeberhan for suggesting how to elucidate the shortest path to the proof of Theorem \ref{THM:Amdeberhan--Leven--conjecture}, an anonymous expert for suggesting several helpful changes, and an anonymous referee for a thorough reading. The author would also like to thank Paul Johnson \cite{Johnson} and Matthew Fayers \cite{F1,Fayers} for their excellent exposition on core partitions, Matthias Beck and Sinai Robins \cite{Ehrhart} for their wonderful exposition on Ehrhart theory, and J\o rn Olsson \cite{Olsson} for a counterexample showing that the $t$-core and $s$-core operations do not generally commute. Finally, the author thanks the \href{https://www.overleaf.com/read/xmzxdgbdcpnq}{Overleaf} website for speeding up the writing process.

\end{document}